\documentclass[12pt,a4paper]{article}
\usepackage[utf8]{inputenc}
\usepackage{mathrsfs,amsthm,xcolor,verbatim,bbm,amsmath,amsfonts,amssymb,nicefrac,enumitem,hyperref,bm,mathtools,xparse,etoolbox,MnSymbol}
\usepackage[margin=1in]{geometry}
\usepackage{algorithm,algpseudocode}
\usepackage[nocompress]{cite}

\usepackage[capitalise,sort]{cleveref}

\crefname{enumi}{item}{items}
\crefname{equation}{}{}
\crefname{figure}{Figure}{Figures}
\crefname{listing}{Source code}{Source codes}
\crefname{lstlisting}{Source code}{Source codes}
\crefname{cor}{Corollary}{Corollaries}

\crefname{subsection}{Subsection}{Subsections}

\hypersetup{
    colorlinks, linktocpage=true,
    linkcolor={blue!80!black},
    citecolor={green},
    urlcolor={blue!80!black}
}

\newcommand{\black}{\color[rgb]{0,0,0}}
\newcommand{\red}{\color[rgb]{1,0,0}}
\renewcommand{\red}{}
\newcommand{\blue}{\color[rgb]{0,0,1}}
\renewcommand{\blue}{}

\newcommand{\eps}{\varepsilon}

\newcommand{\cF}{\mathcal F}

\newcommand{\cU}{\mathcal U}

\newcommand{\var}{\mathrm{Var}}

\newcommand{\dd}{\mathrm{d}}
\newcommand{\sfrac}[2]{\mbox{$\frac{#1}{#2}$}}

\newcommand{\IV}{\mathbb V}
\newcommand{\IX}{\mathbb X}
\newcommand{\IM}{\mathbb M}

\newcommand{\bX}{\mathbf X}

\newcommand{\bas}[1]{\begin{align}\begin{split} #1
\end{split}\end{align}}

\newcommand{\dtilde}[1]{\tilde{\tilde{#1}}}

\newcommand{\1}{1\hspace{-0.098cm}\mathrm{l}}

\renewcommand{\P}{{\mathbb P}}

\newcommand{\N}{{\mathbb N}}
\newcommand{\Z}{{\mathbb Z}}

\newcommand{\E}{{\mathbb E}}

\newcommand{\R}{{\mathbb R}}

\newcommand{\cK}{{\mathcal K}}

\newcommand{\IU}{{\mathbb U}}

\newcommand{\thetai}{\theta^{(i)}}
\newcommand{\sigmai}{\sigma^{(i)}}

\newcommand{\bx}{\mathbf x}

\theoremstyle{plain}
\newtheorem{theorem}{Theorem}[section]
\newtheorem{prop}[theorem]{Proposition}
\newtheorem{lemma}[theorem]{Lemma}

\newtheorem{defi}[theorem]{Definition}

\theoremstyle{definition}
\newtheorem{rem}[theorem]{Remark}

\makeatletter
\@namedef{subjclassname@2020}{%
	\textup{2020} Mathematics Subject Classification}
\makeatother

\ExplSyntaxOn

\bool_new:N \g_forexample

\NewDocumentCommand{\eg}{ o }{
\IfValueT{#1}{
\str_if_eq:noTF {fe} {#1} {
\bool_gset_true:N \g_forexample
} {\bool_gset_false:N \g_forexample}
}
\bool_if:nTF { \g_forexample } {
\bool_gset_false:N \g_forexample
for~example
}{
\bool_gset_true:N \g_forexample
for~instance
}
}

\ExplSyntaxOff

\makeatletter
\ExplSyntaxOn
\seq_new:N \g_abbrs
\prop_new:N \g_abbr_counts
\tl_new:N \l_abbr_count_tl

\NewDocumentCommand{\abbr}{m m O{#1} m m O{#4}}{
	\expandafter\newcommand\csname#3\endcsname[1][]{
		\seq_if_in:NnTF \g_abbrs {#1} {
			\prop_get:NnN \g_abbr_counts {#1} \l_abbr_count_tl
			\prop_gput:Nnx \g_abbr_counts {#1} {\int_eval:n {\l_abbr_count_tl + 1}}
			\hyperref[#1]{#1}
		} {
			\seq_gput_left:Nn \g_abbrs {#1}
			\prop_gput:Nnn \g_abbr_counts {#1} {1}
			\expandafter\gdef\csname#1@def\endcsname{#2}
			\phantomsection\label{#1}
			\str_if_eq:nnTF{##1}{}{\emph{#2}}{##1}~(\hyperref[#1]{#1})
		}
	}
	\expandafter\newcommand\csname#6\endcsname[1][]{
		\seq_if_in:NnTF \g_abbrs {#1} {
			\prop_get:NnN \g_abbr_counts {#1} \l_abbr_count_tl
			\prop_gput:Nnx \g_abbr_counts {#1} {\int_eval:n {\l_abbr_count_tl + 1}}
			\hyperref[#1]{#4}
		} {
			\expandafter\gdef\csname#1@def\endcsname{#5}
			\seq_gput_left:Nn \g_abbrs {#1}
			\prop_gput:Nnn \g_abbr_counts {#1} {1}
			\phantomsection\label{#1}
			\str_if_eq:nnTF{##1}{}{\emph{#5}}{##1}~(\hyperref[#1]{#4})
		}
	}
}

\ExplSyntaxOff
\makeatother

\abbr{ANN}{artificial neural network}{ANNs}{artificial neural networks}
\abbr{SGD}{stochastic gradient descent}{SGDs}{stochastic gradient descent}
\abbr{GD}{gradient descent}{GDs}{gradient descent}
\abbr{SOP}{stochastic optimization problem}{SOPs}{stochastic optimization problems}
\abbr{ODE}{ordinary differential equation}{ODEs}{ordinary differential equations}
\abbr{PDE}{partial differential equation}{PDEs}{partial differential equations}
\abbr{DKM}{deep Kolmogorov method}{DKMs}{deep Kolmogorov methods}
\abbr{PINN}{physics-informed neural network}{PINNs}{physics-informed neural networks}
\abbr{DRM}{deep Ritz method}{DRMs}{deep Ritz methods}
\abbr{GELU}{Gaussian error linear unit}{GELUs}{Gaussian error linear units}
\abbr{ReLU}{rectified linear unit}{ReLUs}{rectified linear units}
\abbr{ResNet}{residual neural network}{ResNets}{residual neural networks}

\title{Convergence rates for the Adam optimizer}

\author{Steffen Dereich$^{1}$ and Arnulf Jentzen$^{2,3}$\bigskip\\
\small{$^1$ Institute for Mathematical Stochastics, University of M\"unster,}\vspace{-0.1cm}\\
\small{Germany; e-mail: \texttt{steffen.dereich}\textcircled{\texttt{a}}\texttt{uni-muenster.de}}\smallskip\\
\small{$^2$ School of Data Science and Shenzhen Research Institute of Big Data,}\vspace{-0.1cm}\\
\small{The Chinese University of Hong Kong, Shenzhen (CUHK-Shenzhen),}\vspace{-0.1cm}\\
\small{China; e-mail: \texttt{ajentzen}\textcircled{\texttt{a}}\texttt{cuhk.edu.cn}}\smallskip\\
\small{$^3$ Applied Mathematics: Institute for Analysis and Numerics,}\vspace{-0.1cm}\\
\small{University of M\"unster, Germany; e-mail: \texttt{ajentzen}\textcircled{\texttt{a}}\texttt{uni-muenster.de}}}

\date{\today}

\begin{document}

\maketitle

\begin{abstract}
Stochastic gradient descent (SGD) optimization methods are nowadays the method of choice for the training of 
deep neural networks (DNNs) in artificial intelligence systems. In practically relevant training problems, 
usually not the plain vanilla standard SGD method is the employed optimization scheme but instead 
suitably accelerated and adaptive SGD optimization methods are applied. As of today, maybe the most popular 
variant of such accelerated and adaptive SGD optimization methods 
is the famous Adam optimizer proposed by 
Kingma \& 
Ba in 2014. 
Despite the popularity of the Adam optimizer in implementations, 
it remained an open problem of research to provide a convergence analysis for the Adam optimizer 
even in the situation of simple quadratic stochastic optimization problems where the objective function 
(the function one intends to minimize) is strongly convex. 
In this work we solve this problem by establishing 
optimal convergence rates for the Adam optimizer for a large class of stochastic optimization problems, 
in particular, covering simple quadratic stochastic optimization problems. 
The key ingredient of our convergence analysis is a new vector field function 
which we propose to refer to as the \emph{Adam vector field}. 
This Adam vector field accurately describes the macroscopic behaviour of the Adam optimization process 
but differs from the negative gradient of the objective function 
(the function we intend to minimize) of the considered stochastic optimization problem. 
In particular, our convergence analysis reveals that the Adam optimizer does 
typically not converge to critical points of the objective function 
(zeros of the gradient of the objective function)
of the considered optimization problem but converges with rates 
to zeros of this Adam vector field. 
Even though our convergence results are only formulated for the Adam optimizer, 
the arguments in our convergence analysis can also be applied to other related SGD optimization methods and 
open the door for a systematic mathematical treatment of a large class of adaptive 
and/or accelerated SGD optimization methods. 
\end{abstract}

%

\tableofcontents

\section{Introduction}
\label{sec:introduction}

\SGD[\emph{Stochastic gradient descent}]\ optimization methods 
are nowadays the method of choice for the training of 
deep \ANNs\ in artificial intelligence systems. In practically relevant training problems, 
usually not the plain vanilla standard \SGD\ method is the employed optimization scheme but instead 
suitably accelerated and adaptive \SGD\ optimization methods are considered, 
such as the momentum SGD (cf.\ \cite{Polyak64}),
the Nesterov accelerated SGD (cf.\ \cite{Nesterov83,sutskever2013importance}), 
the Adagrad (cf.\ \cite{DuchiHazanSinger11}),
the RMSprop (cf.\ \cite{HintonSlides}),
the Adadelta (cf.\ \cite{Zeiler12}),
and the Adam (cf.\ \cite{KingmaBa2024_Adam}) 
optimizers. 
As of today, maybe the most popular variant of such accelerated 
and adaptive \SGD\ optimization methods 
is the famous Adam optimizer proposed in 
Kingma \& Ba \cite{KingmaBa2024_Adam}, 
which is essentially a combination of 
the accelerated momentum SGD optimizer 
and the adaptive RMSprop optimizer. 
We also refer, \eg, to the overview articles \cite{ruder2017overviewgradientdescentoptimization,Sun2019} 
and the monograph \cite[Chapters~6 and 7]{JentzenBookDeepLearning2023}
for details and further references on such 
and related accelerated and adaptive \SGD\ optimization methods.

Despite the popularity of the Adam optimizer in implementations, 
it remained an open problem of research to provide a convergence analysis 
for the Adam optimizer even in the situation of simple 
quadratic \SOPs\ (such as
\cref{eq:example_problem} and \cref{eq:thm-1_innovation} below) 
where the objective function 
(the function one intends to minimize) is strongly convex 
(see also below for a brief overview on results in the literature analyzing accelerated and adaptive \SGD\ optimization methods).

In this work we solve this problem by establishing 
optimal convergence rates for the Adam optimizer for a large class 
of \SOPs\ (cf.\ \cref{thm-2} in \cref{sec:main_results} below for the main result of this work), 
in particular, covering simple quadratic \SOPs. 
To outline the findings of this work, within this introductory section we now present 
in \cref{thm-1} below a consequence of our general convergence analysis in the situation of 
the simple examplary quadratic \SOPs\ to minimize the function 
\begin{equation}
\label{eq:example_problem}
  \R^d \ni \theta 
  \mapsto 
  \E[ \ell( \theta, U^{ (1) } ) ]
  \in \R
\end{equation}
with $ U^{ (1) } \colon \Omega \to \R^d $
being an $ \R^d $-valued random variable on a probability space $ ( \Omega, \cF, \P ) $ 
and with the loss function 
$
  \ell \colon \R^d \times \R^d \to \R
$
being\footnote{Note that 
for all $ n \in \N $, $ x = ( x_1, \dots, x_n ) \in \R^n $
it holds that 
$
  | x | = [ \sum_{ k = 1 }^n ( x_k )^2 ]^{ 1 / 2 }
$ is the standard norm of the vector $ x $.} 
the standard mean square error loss 
$
  \ell( \theta, u ) = | \theta - u |^2
$
for 
$
  \theta, u \in \R^d
$. 
We note that in this simple example the negative gradient of $ \ell $ 
with respect to the parameter space variable satisfes that for all 
$
  \theta, u \in \R^d
$
it holds that 
$  
  - ( \nabla_{ \theta } \ell )( \theta, u ) = u - \theta 
$ 
(see \cref{eq:thm-1_innovation} in \cref{thm-1} below).

\begin{samepage}
\begin{theorem}[Convergence rates for the Adam optimizer in a special case] 
\label{thm-1}
Let 
\begin{enumerate}[label=(\roman*)]
\item $ d \in \N $, $\alpha,\beta\in[0,1)$, $\epsilon\in(0,\infty)$ with $\alpha<\sqrt \beta $, 
\item $U=(U^{(i)})_{i\in\N}$ be a sequence of  bounded  i.i.d.\ $ \R^d $-valued random variables, 
\item 
\label{eq:item_zero_sequence}
$ (\gamma_{n})_{n\in\N} $ be a decreasing 
$(0,\infty)$-valued zero-sequence 
with 
$\lim_{n\to\infty}  \frac {\gamma_{n}- \gamma_{n+1}}{\gamma_{n}^{2}} =0$
and
\item for every $M\in\N$, 
\bas{
\label{eq:thm-1_innovation}
X_M \colon ( \R^d )^\infty\times   \R^d\to\R^d, \  ((u^{(i)})_{i\in\N},\theta)\mapsto \frac 1M\sum_{i=1}^M (u^{(i)} -\theta).
}
\end{enumerate}
Then for every $p\in(0,\infty)$, $\xi\in\R^d$  there exist $M_0\in\N$, $\eta\in(0,\infty)$ and a $\R^d$-valued sequence $(\vartheta_M)_{M\in\N\cap[M_0,\infty)}$ such that for every $M\in\N\cap[M_0,\infty)$ the Adam algorithm 
 $(\theta_n)_{n\in\N_0}$ with innovation $ (X_M, U ) $, damping parameters $(\alpha,\beta,\epsilon)$ and step-sizes $ ( \gamma_n )_{ n \in \N } $ 
 started in $ \xi $ 
(cf.\  \cref{def:Adam} below) satisfies  
\begin{enumerate}[label=(\alph*)]
 \item
\label{item_a}
 $\lim_{n\to\infty} \theta_n=\vartheta_M,$ almost surely,
 \item
\label{item_b}
 $\E[|\theta_n-\vartheta_M|^p]^{1/p}\le \eta\sqrt{ \gamma_n}$, for all $n\in\N$, and
 \item
\label{item_c}
 $|\vartheta_M-\E[U^{(1)}]|\le \eta \,M^{-1}$.
 \end{enumerate}
\end{theorem}
\end{samepage}

Roughly speaking, \cref{thm-1} establishes that the Adam optimizer applied 
to the simple examplary quadratic \SOP\ in \cref{eq:example_problem} 
converges for a sufficiently large mini-batch size $ M \in \N $ 
almost surely to a random limit point $ \vartheta_M $ 
(see \cref{item_a} in \cref{thm-1}), 
shows that the strong $ L^p $-distance 
between the Adam optimization process and the random limit point $ \vartheta_M $ 
is bounded by a constant multiplied with the square root of the learning rate 
(optimal convergence rate $ \nicefrac{ 1 }{ 2 } $ in dependence of the size of the learning rate) 
(see \cref{item_b} in \cref{thm-1}), 
and proves that the random limit point $ \vartheta_M $ converges 
with convergence rate $ 1 $ to the global minimizer 
of the \SOP\ in \cref{eq:example_problem} 
as the size of the mini-batch $ M $ increases to infinity 
(see \cref{item_c} in \cref{thm-1}). 
We note that the assumption in 
\cref{eq:item_zero_sequence} in \cref{thm-1} 
that the sequence of learning rates $ ( \gamma_n )_{ n \in \N } $ 
is a zero-sequence in the sense that 
\begin{equation}
\label{eq:learning_rates_zero_sequence}
  \limsup_{ n \to \infty } \gamma_n = 0 
\end{equation}
is necessary and can not be avoided. Indeed, in \cite{DereichGraeberJentzen2024arXiv_non_convergence} 
it is proved for several \SOPs\ that every component of the Adam optimizer fails to converge to any possible random point 
if the learning rates do not decay to zero.

\cref{thm-1} is a consequence of the more general convergence analysis for the Adam algorithm 
that we develop within this article. To be more specific, in the main result of this article, 
\cref{thm-2} in \cref{sec:main_results} below, we provide an \ODE\ based 
convergence analysis for the Adam optimizer and, in praticular, we establish 
optimal convergence rates for the Adam optimizer for a large class of \SOPs, 
in particular, covering simple quadratic \SOPs\ (such as \cref{eq:example_problem} above).

The key ingredient of our convergence analysis is a new vector field function 
which we propose to refer to as the \emph{Adam vector field} (cf.\ \cref{def:Adamfield0} below). 
This Adam vector field accurately describes the macroscopic behaviour of the Adam optimization process 
but differs from the negative gradient of the objective function 
(the function we intend to minimize) of the considered \SOP\ (cf.\ \cref{def:Adamfield0,thm-2}). 
In particular, our convergence analysis reveals that the Adam optimizer does 
typically not converge to critical points of the objective function 
(zeros of the gradient of the objective function)
of the considered optimization problem 
(cf.\ \cref{def:linApp}, 
\cref{thm:perturbation_analysis}, 
and \cref{rem:non-convergence}
in \cref{sec:perturbation_analysis} below) 
but converges with rates to zeros of this Adam vector field 
(cf.\ \cref{thm-2}). 
Loosely speaking, in \cref{thm-2} 
we provide for every sufficiently large $ n \in \N $ 
an upper bound for the strong $ L^p $-distance 
between the Adam optimization process at time $ n \in \N $
and the solution process of the \ODE\ with the vector field in the \ODE\ given 
by the Adam vector field 
(cf.\ \cref{eq:Lp_estimate_ODE_Adam} in \cref{thm-2}). 
In particular, if the solution process of the \ODE\ starts in 
a zero of the Adam vector field, 
then the solution process of the \ODE\ reduces 
to a constant function and, in such a situation, 
we have that 
\cref{thm-2} provides 
for every sufficiently large $ n \in \N $ 
an upper bound for the strong $ L^p $-distance 
between the Adam optimization process at time $ n \in \N $ 
and this zero of the Adam vector field 
(cf.\ \cref{eq:Lp_estimate_ODE_Adam} in \cref{thm-2}). 
Even though our convergence results in \cref{thm-1,thm-2} are only 
formulated for the Adam optimizer, 
the arguments in our convergence analysis can also be applied to other SGD optimization methods 
and open the door for a systematic mathematical treatment of a large class of possibly adaptive 
and/or accelerated SGD optimization methods.

\subsection{Literature review}
In the following we supply a brief survey on selected works 
in the literature providing error or regret analyses for 
the Adam optimizer or 
related 
\SGD\ based optimization procedures.

First, we refer, \eg, to 
\cite{CHERIDITO2021101540,jentzen2024nonconvergenceglobalminimizersadam,Lu_2020,gallon2022blowphenomenagradientdescent,ReddiKale2019}
for non-convergence/divergence results or lower error or regret bounds 
for the Adam optimizer and related gradient based optimization procedures. 
In particular, the work Reddi et al.\ \cite[Theorems 1, 2, and 3]{ReddiKale2019} shows 
that for any choice of the damping paramters $ \alpha, \beta \in [0,1) $
with $ \alpha < \sqrt{ \beta } $ there exists a 
convex stochastic optimization problem 
in which (a slightly modified variant of) the Adam optimizer with the damping parameters $ \alpha $ and $ \beta $ 
does not converge to the optimal solution.

Moreover, in the work \cite{DereichGraeberJentzen2024arXiv_non_convergence} 
it is shown for several \SOPs\ that every component of the Adam optimizer 
and related \SGD\ optimization methods fail to converge 
to any possible scalar random variable 
if the learning rate do not decay to zero 
(cf.\ \cref{eq:learning_rates_zero_sequence} above).

Furthermore, in the situation of the training of shallow \ANNs\ with the \ReLU\ activation 
the paper \cite{jentzen2024nonconvergenceglobalminimizersadam} 
proves, in particular, that a class of \SGD\ optimization methos including 
the Adam optimizer as a special case fails to converge to \emph{global minimizers} 
in the ANN optimization training landscape 
(cf.\ also \cite{CHERIDITO2021101540,Lu_2020,gallon2022blowphenomenagradientdescent,MR4243432}). 
The findings in \cite{jentzen2024nonconvergenceglobalminimizersadam} (and \cite{CHERIDITO2021101540,Lu_2020}, respectively) 
do, however, not exclude the possibility that the Adam optimizer converges to a good non-optimal local minimizer 
whose value under the objective function is close to the infimal value of the objective function 
(cf.\ also \cite{GentileWelper2022arXiv,Ibragimovetal2022arXiv,Welper2023arXiv,Welper2024arXiv} for convergence results 
to good non-optimal critical points in simplified shallow \ANN\ training setups).

Moreover, we refer, \eg, to 
\cite{GodichonBaggioni2023,DereichJentzenRiekert2024arXiv_adaptive,
ZhangChen2022,
Barakat_2021_cvg,
li2023convergenceadamrelaxedassumptions,
ReddiKale2019,
hong2024revisitingconvergenceadagradrelaxed,ZouShen2019,Defossez2022,DereichKassing2021arXiv,HeLiangLiuXu2023arXiv,
DuchiHazanSinger11,KingmaBa2024_Adam,Chenetal2018,
ZhangChen2022,ZouShen2019,DingXiaoToh2023arXiv,MR4723898,Huetal2024}
for works providing convergence results or 
(possibly asymptotic) upper error or regret bounds 
for Adam and related \SGD\ based optimization procedures. 
For example, in Reddi et al.\ \cite[Algorithm~2]{ReddiKale2019}
a modified variant of the Adam optimizer 
(referred to as AMSGrad optimizer), 
in which, loosely speaking, the componentwise maximum 
over previously calculated componentwise squared gradients 
is incorporated (so that the non-convergence issue in \cite[Theorems 1, 2, and 3]{ReddiKale2019} can be avoided), 
is proposed and studied. 
In particular, the work \cite[Theorem 4 and Corollary 1]{ReddiKale2019} 
provides upper bounds for the regret of this modified variant 
of the Adam optimizer.

Furthermore, in the situation of suitable strongly convex \SOPs\ the article 
Godichon-Baggioni \& Tarrago~\cite[Theorems~3.1 and 3.2 and Propositions~3.1 and 3.2]{GodichonBaggioni2023} 
provides an error analysis for a general abstract class of adaptive \SGD\ optimization methods, 
which is then specialized for the Adagrad optimizer 
(cf.\ \cite[Theorems~3.4, 4.2, and 5.2]{GodichonBaggioni2023})
and stochastic Newton algorithms 
(cf.\ \cite[Theorems~3.3, 4.1, and 5.1 and Corollary~3.1]{GodichonBaggioni2023}).

Moreover, in the situation of \SOPs\ fulfilling coercivity-type conditions 
the work Dereich et al.~\cite[Corollary~4.11]{DereichJentzenRiekert2024arXiv_adaptive} 
establishes convergence to the optimal solution of the considered \SOP\ for 
certain \SGD\ optimization methods with random learning rates. 
In the case of certain simple quadratic \SOPs\ (cf.\ \cref{eq:example_problem} above) 
the article \cite[Theorem~7.1]{DereichJentzenRiekert2024arXiv_adaptive} 
also establishes convergence to the optimal solution of the considered \SOP\ for 
the \SGD\ method with suitable adaptive learning rates. 

Furthermore, in Dereich \& Kassing~\cite{DereichKassing2021arXiv} 
convergence of possibly accelerated \SGD\ optimization methods 
such as the momentum SGD optimizer to a critical point is shown 
(cf., \eg, also \cite[Section~3]{Huetal2024})
under the assumption that the objective function 
satisfies the Kurdyka--\L ojasiewicz inequality.

An \ODE\ based convergence result for the Adam algorithm 
of a different type than in \cref{thm-2} below 
is given in the work 
Barakat \& Bianchi~\cite{Barakat_2021_cvg}. 
Specifically, in \cite{Barakat_2021_cvg} it is, 
among other things, shown 
that the Adam optimizer with constant learning rates 
$ \forall \, n, m \in \N \colon \gamma_n = \gamma_m $ 
(cf.\ \cref{eq:item_zero_sequence} in \cref{thm-1} above) 
converges to the solution of a suitable non-autonomous 
\ODE\ (see (3.3) and (ODE) in \cite{Barakat_2021_cvg}) 
as the constant learning rate $ \gamma_1 $ 
converges to zero. We note that the Adam vector field 
in this work (see \cref{def:Adamfield0} below) 
is autonomous and significantly different 
to the non-autonomous vector field in 
\cite[(3.3) and (ODE)]{Barakat_2021_cvg}.

We also refer, \eg, to 
\cite[Theorem~1, Corollary~1, Corollary~2, Corollary~3, Theorem~5, and Theorem~6]{MR4723898} 
and 
\cite[Theorem~3.8, Corollary~3.11, Theorem~4.8, Theorem~5.2, and Theorem~5.3]{DingXiaoToh2023arXiv}
and the references therein 
for convergence analyses  
for Adam-type optimizers in the situation of 
non-smooth objective functions. 
In this context, we also refer, \eg, to \cite[Theorem~4.1, Theorem~4.2, and Theorem~6.2]{li2023convergenceadamrelaxedassumptions} 
(and the references mentioned therein) 
for upper regret bounds for the Adam optimizer 
and a variance-reduced variant of the Adam optimizer 
(referred to as VRAdam optimizer) 
under less restrictive global boundedness/growth 
assumptions on the gradient of the objective function.

For further references and more detailed reviews on the literature 
on Adam and related \SGD\ optimization methods 
we also refer, \eg, to \cite{ruder2017overviewgradientdescentoptimization,MR4205063,Eetal2020survey,Sun2019,JentzenBookDeepLearning2023}.

\subsection{Structure of this article}

The remainder of this article is organized as follows. 
In \cref{sec:main_results} we introduce the Adam algorithm in full details 
(cf.\ \cref{def:V} and \cref{def:Adam} below), 
we introduce our concept of the Adam vector field 
(cf.\ \cref{def:Adamfield0} below), 
we present in \cref{thm-2} below the main convergence result of this work, 
and we introduce a suitable space of vector valued sequences, 
which we refer to as Adam sequence space, 
used to formulate \cref{thm-2} (cf.\ \cref{def:spaces} below).

Our main goal in \cref{sec:innovation_Lipschitz} is, loosely speaking, 
to establish that the innovation in the Adam optimizer (cf.\ \cref{def:V})
considered as a function on the Adam sequence space 
is a uniformly bounded non-expansive mapping (weak contraction) 
on the Adam sequence space (cf.\ \cref{le:23456} below).

In \cref{sec:Adam_evolution_with_fixed_theta} 
we analyze the evolution of a simplified variant of 
the Adam optimizer in which we do not employ  
the current state of the Adam optimizer in the recursions for the momentum 
and daming terms 
(cf.\ \cref{def:Adam})
but in which we instead employ a fixed 
predetermined external vector $ \theta \in \R^d $ 
in the recursions for the momentum and daming terms 
(cf.\ \cref{le:842,le:1345} in \cref{sec:Adam_evolution_with_fixed_theta}).

A key idea in our proof for convergence rates 
for the Adam optimizer (cf.\ \cref{thm-1} above and \cref{thm-2} below) is to not study Adam optimization 
processes directly but instead, first, to analyse suitable approximations 
of Adam optimization processes 
and, thereafter, to estimate 
the differences between the Adam optimization processes 
and their approximations. 
Our main subject in \cref{sec:approx_intro} 
is precisely to study such approximations of 
the Adam optimization processes 
(cf.\ \cref{eq:approximation_process_1}, 
\cref{eq:approximation_process_2}, 
\cref{prop:234}, and \cref{le:3571}
in \cref{sec:approx_intro}). 
Our proof of \cref{prop:234} 
(the main result of \cref{sec:approx_intro}) 
employs applications of the results established 
in \cref{sec:Adam_evolution_with_fixed_theta}, 
that is, we employ \cref{le:842,le:1345} from \cref{sec:Adam_evolution_with_fixed_theta} 
in our proof of \cref{prop:234}.

In \cref{sec:technical_estimates} 
we apply the findings of \cref{sec:approx_intro} 
(cf.\ \cref{prop:234,le:3571,le:234787} in \cref{sec:approx_intro}) 
to establish \ODE\ based non-uniform 
(cf.\ \cref{eq24811-0} in \cref{prop35-12}) 
and uniform (cf.\ \cref{eq:uniform_error_analysis} in \cref{prop35-2})
error analyses for the Adam optimizer.

In \cref{sec:proof_main_result} 
we combine the findings 
from \cref{sec:innovation_Lipschitz,sec:technical_estimates} 
(cf.\ \cref{le:23456} in \cref{sec:innovation_Lipschitz} and 
\cref{prop35-12} and \cref{prop35-2} in \cref{sec:technical_estimates}) 
to complete the proof of our main convergence rate result 
for the Adam optimizer in \cref{thm-2} below.

In \cref{sec:perturbation_analysis} we establish 
regularity properties 
(cf.\ \cref{theo:732846} and \cref{le:2341})
and a perturbation analysis  
(cf.\ \cref{thm:perturbation_analysis})
for the Adam vector field.

In \cref{sec:proof_theorem_introduction} 
we combine the regularity properties established in \cref{sec:perturbation_analysis} 
(cf.\ \cref{theo:732846} and \cref{le:2341})
with an application of \cref{thm-2} 
to the simple quadratic \SOP\ in \cref{eq:example_problem} 
to deliver the proof of our specialized 
convergence rate result for the Adam optimizer in \cref{thm-1} above.

\section{Adam sequence space, Adam vector field and main convergence result}
\label{sec:main_results}

In this section, we introduce the central objects and the main theorems of this article. 
For this we denote by $(\Omega,\cF,\P)$ a probability space and by $d\in\N$ the dimension of the underlying problem.
Let us introduce the Adam algorithm with full details. 
 
\begin{defi}[Innovation]\label{def:V}Let $\cU$ be a measurable space.
 A pair $(X,U)$ consisting of
\begin{enumerate}[label=(\roman*)]
\item a $\cU$-valued random variable $U$ and
\item a product measurable mapping $X \colon \cU\times \R^{d}\to \R^{d}$ 
\end{enumerate}
is called \emph{innovation}.
\end{defi}

\begin{defi}[Adam optimizer]\label{def:Adam}Let
\begin{enumerate}[label=(\roman*)]
\item $\alpha,\beta\in[0,1)$, $\epsilon\in(0,\infty)$ with $\alpha<\sqrt \beta$  (the \emph{damping parameters}),
\item $n_0\in\N_{0}$, $\theta_{n_{0}},m_{n_{0}}\in\R^{d}$, $v_{n_{0}}\in[0,\infty)^d$ (the \emph{initialisation}),
\item a decreasing $(0,\infty)$-valued sequence $(\gamma_{n})_{n\in\N}$ (the \emph{sequence of step-sizes}) and
\item $(X,U)$ an innovation.
\end{enumerate}
%
%
An $\R^{d}$-valued stochastic process  $(\theta_{n})_{n\in\N_0\cap[n_{0},\infty)}$ is called \emph{Adam algorithm with damping parameters $(\alpha,\beta,\epsilon)$ and step-sizes $(\gamma_{n})$ started at time $n_{0}$ in $(\theta_{n_0},m_{n_{0}},v_{n_{0}})$}, if it satisfies for every $n\in\N\cap(n_{0},\infty)$ and $i=1,\dots,d$
\bas{\label{eq:Adam}
\thetai_n=\thetai_{n-1}+\gamma_{n}\, \sigmai_{n}\,  m^{(i)}_n,
}
where $(U_{n})_{n\in\N\cap(n_{0},\infty)}$ is a family of independent copies of $U$ and  for $n\in\N\cap(n_{0},\infty)$ and $i=1,\dots,d$,
\begin{itemize}
\item  $m_n=\alpha \,m_{n-1}+(1-\alpha) \,X(U_{n},\theta_{n-1})$,
\item  $v^{(i)}_n= \beta \,v^{(i)}_{n-1} +(1-\beta) \,(X^{(i)}(U_{n},\theta_{n-1}))^{2}$ and
\item  $\displaystyle \sigma^{(i)}_n=\frac1{\sqrt {v_n^{(i)}/(1-\beta^{n})}+\epsilon}$.
\end{itemize}
We call the sequence $(t_{n})_{n\in\N_{0}}$ given by
\bas{\label{def:tn}
t_{n}=\sum_{k=1}^{n} \gamma_{k}
}
the \emph{training times} of the Adam algorithm.
\end{defi}

Technically, the Adam algorithm defines a Markov chain when considering the tuple $(\theta_{n},m_{n},v_{n})_{n\ge n_{0}}$ as the process of states. When doing so one faces the difficulty that  the $\sqrt {v_n^{(i)}}$-term in the definition of $\sigma_{n}^{(i)}$ is not Lipschitz which makes it hard to analyse the system. We bypass that problem by interpreting the Adam algorithm as a delay equation where the evolution is governed by the whole history of innovations that are considered as an element of an infinite dimensional sequence space which we refer to as Adam sequence space. 

\begin{defi}[Adam sequence space]
\label{def:spaces}
Let $ \alpha, \beta \in [0,1) $, $ \epsilon \in (0,\infty) $ 
with $ \alpha < \sqrt{ \beta } $ be damping factors. 
We let \begin{align}\label{def:varrho}
\varrho_{k}=(1-\alpha) \epsilon^{-1}\Bigl(\alpha^{-k}+ 
\frac{1}{\sqrt{1-\alpha^{2}/\beta}}\beta^{-k/2}\Bigr),\qquad\text{ for all $k\in-\N_{0}$},
\end{align}
and denote by $\ell^{d}_{\varrho}$ the space of all sequences $\mathbf x=(x_{k})_{k\in-\N_{0}}\in(\R^{d})^{-\N_{0}}$ with
\bas{
\|\mathbf x\|_{\ell_{\varrho}^{d}}:=\sum_{k\in-\N_{0}} \rho_{k}\, |x_{k}|<\infty.
}
We equip the space with  the norm $\|\cdot\|_{\ell_{\varrho}^{d}}$. 
Moreover, we let for  a tuple $(m,v)\in\R^{d}\times [0,\infty)^{d}$
\bas{
(m,v)_{\ell_{\varrho}^{d}}=\inf\Bigl \{ \|\bx\|_{\ell_{\varrho}^{d}}: \bx\in \ell_{\varrho}^{d},\ &m=(1-\alpha) \sum_{k\in-\N_{0}} \alpha^{-k} x_{k} \text{ \ and } \\
&v^{(i)}=(1-\beta) \sum_{k\in-\N_{0}} \beta^{-k} (x_{k} ^{(i)})^{2}\text{ for }i=1,\dots,d\Bigr\},
}
where the infimum of the empty set is $\infty$.
\end{defi}

Our main result is an ODE approximation for the Adam algorithm. The respective vector field driving the ODE is defined in the following definition.

\begin{defi}[Adam vector field] 
\label{def:Adamfield0} For a tuple $(\alpha,\beta,\epsilon)$ of damping parameters  and an innovation $(X,U)$ we call the mapping
$
f:\R^{d}\to\R^{d}
$
given by
\bas{\label{def:Adamfield}
f^{(i)}(\theta)=(1-\alpha) \,\E^{\theta}\Bigl[ 1/\Bigl(\sqrt{(1-\beta) \sum_{k\in-\N_0} \beta^{-k} X^{(i)}(U_{k},\theta)^{2}}+\epsilon\Bigr)  \sum_{k\in-\N_0} \alpha^{-k} X^{(i)}(U_k,\theta)\Bigr]
}
with $(U_{k})_{k\in-\N_{0}}$ being a family of independent copies of $U$, the \emph{Adam vector field of the innovation $(X,U)$ for the damping parameters $(\alpha,\beta,\epsilon)$}.
\end{defi}

\begin{theorem}[Convergence rates for the Adam optimizer]  
\label{thm-2}
Let 
\begin{enumerate}[label=(\Roman*)]
\item $\alpha,\beta\in[0,1)$ with $\alpha<\sqrt \beta$ and $\epsilon\in(0,\infty)$ (the \emph{damping parameters}),
\item a decreasing $(0,\infty)$-valued zero-sequence $(\gamma_{n})_{n\in\N}$ (\emph{sequence of step-sizes}), 
\item $\cK,c_{1},c_{2}\in[0,\infty)$ and $p\in(2,\infty)$,
\item {\red $\|\cdot\|$ a norm on $\R^{d}$ being induced by a scalar product $\llangle\cdot,\cdot\rrangle$} 
\end{enumerate}
and suppose that
\bas{\label{eq87236455}
\limsup_{n\to\infty}\frac { \gamma_{n}- \gamma_{n+1}}{ \gamma_{n}^{2}}\le 2c_{2}<2c_{1}.
}
Then for every $c\in (0,c_{1}-c_{2})$ and $\varepsilon\in(0,\infty)$ there exist $\mathfrak n\in\N_{0}$ and $\eta\in(0,\infty)$ so that the following statement is true.\smallskip

\begin{enumerate}[label=(\roman*)]
\item \emph{Innovation.} Let $(X,U)$ be  an innovation and $V\subset \R^{d}$ a measurable set such that for every $\theta,\theta'\in V$
\bas{\label{def:ellrho} \E[|X(U,\theta)|^{p}]^{1/p}\le \cK\text{ \ and \ } \E[|X(U,\theta)-X(U,\theta')|^{p}]^{1/p} \le \cK  \,|\theta-\theta'|. 
}
\item \emph{Adam algorithm.} Let $n_{0}\in\N_{0}\cap [\mathfrak n,\infty)$ and $(\theta_{n})_{n\ge n_{0}}$ be an Adam algorithm with innovation $(X,U)$ started at time $n_{0}$ in a state $(\theta_{n_0},m_{n_{0}},v_{n_{0}})\in\R^{d}\times\R^{d}\times [0,\infty)^{d}$ so that 
\bas{\label{eq34675}
\sqrt{\gamma_{n_{0}+1}} \,(m_{n_{0}},v_{n_{0}})_{\ell_{\varrho}^{d}}\le \cK.
}
\item \emph{ODE.} 
Let $f:\R^d\to\R^d$ the Adam vector field of the innovation $(X,U)$ for the damping parameters $(\alpha,\beta,\epsilon)$  and $\Psi:[t_{n_{0}},\infty)\to\R^{d}$ a solution to the ODE
$$
\dot \Psi_{t} = f(\Psi_t)
$$
staying in $V\subset\R^{d}$ and satisfying
 \bas{\label{eq87345}
 \|\theta_{n_{0}}-\Psi_{t_{n_{0}}}\| \le\cK.
}

\item \emph{Local monotonicity of ODE.}
Let $(\mathfrak R_{t})_{t\ge t_{n_{0}}}$ be a $(0,\infty]$-valued function so that one has, for all $t\in[ t_{n_{0}},\infty)$ and $x\in V\cap \overline{B(\Psi_{t},\mathfrak R_{t})}$ , that
\bas{{\red
\llangle f(x)-f(\Psi_{t}),x-\Psi_{t}\rrangle\le -c_{1} \|x-\Psi_{t}\|^{2}.}
}
\item Let $(R_{s})_{s\ge t_{n_{0}}}$ be a decreasing $(0,\infty)$-valued function such that for every $s\ge t_{n_{0}}$
\bas{
R_{s}/R_{s+1}\le \cK. 
}
\end{enumerate}
Let
\bas{
\mathfrak N=\inf\{n\ge n_{0}: \theta_{n}\not\in V \text{ \ or \ }{\red \|\theta_{n}-\Psi_{t_{n}}\|>\mathfrak R_{t_{n}}}\}.
}
One has, for all $n\in \N_{0}\cap [n_{0},\infty)$, that
\bas{
\label{eq:Lp_estimate_ODE_Adam}
\E[\1_{\{\mathfrak N\ge n\}}{\red \|\theta_{n}-\Psi_{t_{n}}\|^{p}}]^{1/p}\le \Bigl(\eta +\Bigl((1+\varepsilon)\frac{{\red \|\theta_{n_{0}}-\Psi_{t_{n_{0}}}\|}}{\sqrt{\gamma_{n_{0}+1}}}+\eta \sqrt{(m_{n_{0}},v_{n_{0}})_{\ell_{\varrho}^{d}}}\Bigr) e^{-c (t_{n}-t_{n_{0}})}\Bigr)\sqrt{\gamma_{n+1}}
}
and 
{\red\bas{
\E\Bigl[&\Bigl( \sup_{n:n_{0}\le n\le \mathfrak N}   \frac{\|\theta_{n}-\Psi_{t_{n}}\|}{R_{t_{n}}}\Bigr)^{p} \Bigr]\\
&\le \eta \int_{t_{n_{0}}}^{\infty} R_{s}^{-p} \Bigl(1+\Bigl(\frac{\|\theta_{n_{0}}-\Psi_{t_{n_{0}}}\|}{\sqrt{\gamma_{n_{0}+1}}}+\sqrt{(m_{n_{0}},v_{n_{0}})_{\ell_{\varrho}^{d}}}\Bigr) e^{-c(s-t_{n_{0}})} \Bigr)^{p} \,\Gamma_{s}^{(p-1)/2}\,\dd s,
}}
where
$(\Gamma_{t})_{t>0}$ is the real-valued function given by
$ \Gamma_{t}=\gamma_{n} $, for $n\in\N$ and $t\in(t_{n-1},t_{n}]$.
\end{theorem}

\section{The innovation viewed as a non-expansive mapping on the Adam sequence space}
\label{sec:innovation_Lipschitz}

Let $ \alpha, \beta, \epsilon \in [0,\infty) $ with $ \alpha < \sqrt{\beta} < 1 $ 
and let  $(\varrho_{k})_{k\in-\N_{0}}$ as in \cref{def:varrho}.
In this section our main goal is to analyse the mapping
\bas{\label{eq872346}
g \colon \ell_{\varrho}^{d}\to\R^d, 
\qquad  
  \mathbf x=( x_k )_{ k \in -\N_0 }
  \mapsto 
  \Biggl( 
    \frac{
      (1-\alpha) \sum_{ k \in - \N_0 } \alpha^{ - k } x^{ (i) }_k
	}{
	  \epsilon 
	  + 
	  \sqrt{ 
	    (1 - \beta) \sum_{ k \in - \N_0 } \beta^{ - k } ( x^{(i)}_k )^2
	  }
	}
  \Biggr)_{ i \in \{ 1, \dots, d \} }
}
(see \cref{le:23456} below). 
First, we note that $g$ is well defined since for every $\mathbf x\in\ell_{\varrho}^d$ 
and $i\in\{1,\dots,d\}$ we have that the series in the nominator in \cref{eq872346} is summable. 
The denominator in \cref{eq872346} 
may attain the value $\infty$ in which case $g$ is defined to be zero.

\begin{lemma}
\label{le:23456}
Let $ \alpha, \beta, \epsilon \in [0,\infty) $ with $ \alpha < \sqrt{\beta} < 1 $ 
and let $ ( \varrho_k )_{ k \in - \N_0 } $ 
and  $g$ as defined in \cref{def:varrho} and \cref{eq872346}, respectively.  Then $g$ is Lipschitz continuous with Lipschitz constant $1$ and $ g $ is uniformly bounded by $\sqrt d \frac {1-\alpha}{\sqrt {1-\beta}}\frac1{\sqrt{1-\alpha^{2}/\beta}} $.
\end{lemma}

\begin{proof} 
We first assume that $d=1$ and briefly write $\ell_{\varrho}=\ell_{\varrho}^{1}$ and $\|\cdot\|_{\ell_{\varrho}}=\|\cdot\|_{\ell_{\varrho}^{1}}$. 
For $\mathbf x\in\ell_{\varrho}$, let
\bas{\label{eq:78235653}
m(\bx)=(1-\alpha) \sum_{k\in -\N_{0}} \alpha^{-k} x_{k}
\text{ \ and \ }
v(\bx)=(1-\beta) \sum_{k\in -\N_{0}} \beta^{-k} x_{k}^{2}
}
so that 
\bas{
g(\mathbf x)=\frac{m(\bx)}{\sqrt {v(\bx)}+\epsilon}.
}
Using Cauchy-Schwarz and the fact that $\alpha<\sqrt \beta$ we get that for all $\bx\in\ell_{\varrho}$
\begin{align}\begin{split}\label{mv:est}
|m(\bx)|&=\Bigl|(1-\alpha) \sum_{k=-\infty}^{0} (\alpha/\sqrt\beta)^{-k}   \sqrt{\beta}^{-k} x_{k}\Bigr|\le (1-\alpha)\Bigl(\sum_{k=0}^{\infty} (\alpha/\sqrt\beta)^{2k}\Bigr)^{1/2} \Bigl(\sum_{k=-\infty}^{0}  \beta^{-k} x_{k}^{2} \Bigr)^{1/2} \\
&\le \frac {1-\alpha}{\sqrt {1-\beta}}\frac1{\sqrt{1-\alpha^{2}/\beta}} \sqrt {v(\bx)}
\end{split}\end{align}
which entails the uniform bound for $g$ that is stated in the lemma for $d=1$. The multivariate estimates follows immediately by applying this estimate onto the individual components.

To prove the Lipschitz continuity we again focus at first on the case $d=1$. Fix $n\in\N$ and let $\bx=(x_{0},\dots,x_{-n},0,\dots)\in\ell_{\varrho}$. One has for all $k=0,\dots,-n$ that
\bas{
\partial _{x_{k}} g(\bx) = (1-\alpha)\alpha^{-k}\frac1{\sqrt{v(\bx)}+\epsilon}- (1-\beta)\beta^{-k}\frac {m(\bx)}{(\sqrt {v(\bx)}+\epsilon )^{2}}\frac1{\sqrt {v(\bx)}} x_{k}
}
so that by using (\ref{mv:est}) and $\sqrt {v(\bx)}\ge \sqrt{1-\beta}\beta^{-k/2}|x_{k}|$ we get that
\bas{\label{eq:78345534}
|\partial _{x_{k}} g(\bx)| \le (1-\alpha)\epsilon^{-1}\alpha^{-k}+ 
\frac{1-\alpha}{\sqrt{1-\alpha^{2}/\beta}}\epsilon^{-1}\beta^{-k/2}=\varrho_{k}.
}
For another $\tilde \bx=(\tilde x_{0},\dots,\tilde x_{-n},0,\dots)\in\ell_{\varrho}$ we let, for $\ell=0,\dots,n+1$,
\bas{\bx^{\ell}=(\tilde x_{0},\dots,\tilde x_{-\ell+1},x_{-\ell},\dots,x_{-n},0,\dots).}
Note that $\bx^{0}=\bx$ and $\bx^{n+1}=\tilde\bx$ and we get with the triangle inequality and the mean value theorem that
\bas{\label{eq:43985638}
|g(\bx)-g(\tilde \bx)|\le \sum_{\ell=0}^{n} |g(\bx^{\ell})-g(\bx^{\ell+1})|\le \sum_{\ell=0}^{n} \varrho_{-\ell} |x_{-\ell}-\tilde x_{-\ell}|=\|\bx-\tilde\bx\|_{\ell_{\varrho}}.
}
This inequality remains true for general $\bx,\tilde\bx\in\ell_{\varrho}$ since
\bas{
g(\bx)=\lim_{n\to\infty} g(x_{0},\dots,x_{-n},0,\dots).
}
Indeed, this entails that
\bas{
|g(\bx)-g(\tilde \bx)|\le \limsup_{n\to\infty} |g(x_{0},\dots,x_{-n},0,\dots)-g(\tilde x_{0},\dots,\tilde x_{-n},0,\dots)|\le \|\bx-\tilde\bx\|_{\ell_{\varrho}}.
}

\red
To obtain the result in the multivariate setting we proceed as above. Fix $n\in\N$,  $\bx=(x_{0},\dots,x_{-n},0,\dots)\in\ell_{\varrho}^d$ and $\tilde \bx=(\tilde x_{0},\dots,\tilde x_{-n},0,\dots)\in\ell_{\varrho}^d$ and define for every  $\ell=0,\dots,n+1$,
 $\mathbf x^\ell\in \ell_\varrho^d$ accordingly. Then one has for every  $\ell=0,\dots,n$ that
\bas{
|g(\bx^{\ell})-g(\bx^{\ell+1})|= \Bigl(\sum_{i=1}^d |g(\bx^{\ell,(i)})-g(\bx^{\ell+1,(i)})|^2\Bigr)^{1/2}\le \varrho_{-\ell} \Bigl(\sum _{i=1}^d |x_\ell^{(i)}-\tilde x_\ell^{(i)}|^2\Bigr)^{1/2},
}
where we applied~(\ref{eq:78345534}) for each of the $d$ components.
This entails the multivariate variant of estimate~(\ref{eq:43985638}) that is
\bas{
|g(\bx)-g(\tilde \bx)|\le \|\bx-\tilde \bx\|_{\ell^d_\varrho}.
}
The rest follows in complete analogy to the one dimensional setting.
\end{proof}

\begin{lemma}
\label{le:7357}For every $k\in-\N_{0}$, 
one has that $\varrho_{k-1}\le \sqrt\beta \varrho_{k}$ and
the translation operator
\bas{
T \colon \ell_{\varrho}^d\to\ell_{\varrho}^d, 
\qquad 
  \bx=(x_{k})_{k\in-\N_{0}}\mapsto (\1_{\{k\not=0\}} x_{k+1})_{k\in-\N_{0}}
}
is Lipschitz continuous with Lipschitz constant $\sqrt\beta<1$. In particular, it maps $\ell_{\varrho}$ to $\ell_{\varrho}$.
\end{lemma}

\begin{proof}By definition of $(\varrho_{k})$ it immediately follows that, for every $k\in-\N_{0}$,
\bas{
\varrho_{k-1}\le \sqrt \beta \varrho_{k}.
}
Consequently, one has that, for every $\bx\in\ell_{\varrho}^d$,
\bas{
\|T(\bx)\|_{\ell_{\varrho}^d}=\sum_{k\in-\N_{0}} \varrho_{k-1} |x_{k}|\le \sqrt \beta \sum_{k\in-\N_{0}} \varrho_{k} |x_{k}|=\sqrt \beta \,\|\bx\|_{\ell_{\varrho}^d}.
} 
The statement follows by linearity of $T$.
\end{proof}

%

\section{Analysis of the Adam evolution with fixed $\theta$-parameter}
\label{sec:Adam_evolution_with_fixed_theta}
\red

Let us introduce the setting of this section. Let $(X,U)$ be an innovation and $(\alpha,\beta,\epsilon)$ be damping parameters with $0\le \alpha<\sqrt\beta<1$. The Adam algorithm for the innovation $(X,U)$ is a stochastic process where in its iterative definition in \cref{eq:Adam} 
the momentum and damping terms are updated by using the current $\theta$-value of the process. In this section, we analyse a similar process where we keep the $\theta$-value that is used when updating the momentum and damping term fixed.

In the following, fix $ \theta \in \R^d $ and let $(U_k)_{k\in\Z}$ be an i.i.d.\ sequence of copies of $U$. Moreover, for every $n\in\Z$, let $X_n=X(U_n,\theta)$ and
$\bX(n)=(X_{n+k})_{k\in-\N_{0}}$. We fix $n_0\in\N_0$ and consider the process
$(\bar\Theta_{n})_{n\ge n_{0}}$ given by
\bas{
\bar\Theta_{n}=\sum_{k=n_{0}+1}^{n} \gamma_{k}\, g(\bX(k)).
}
\blue

\begin{lemma}\label{le:842} 
Let $p\in[2,\infty)$. One has for every
 $n\in\N_{0}$ with $n\ge n_{0}$ that
\bas{
\E\bigl[|\bar\Theta_{n}-\E[\bar\Theta_n]|^{p}\bigr]^{2/p} \le \kappa^{2} \sum_{k=n_{0}+1}^{n} \gamma_k^{2} \,\E\bigl[|X(U,\theta)-\E[X(U,\theta)]|^{p}\bigr]^{2/p},
}
where
\bas{\label{eq:72635984}
\kappa=2  C_{p}\bigl((1-\sqrt\beta)^{-2} \varrho_{0} \|\varrho\|_{\ell_{1}}+ \|\varrho\|_{\ell_{1}} ^{2}\bigr),
}
$\varrho$ is as in~(\ref{def:varrho}) and $C_{p}$ is the constant appearing in the the Burkholder-Davis-Gundy inequality for the $p$-th moment.
\end{lemma}

\begin{proof}For ease of notation we prove the statement for $d=1$. The generalisation to general $d$ is straight-forward. 
Let $\tilde X_{0}$ be identically distributed as $X_{0}$ and independent of the sequence $(X_{n})_{n\in\Z}$ and set 
$\tilde X_{n}=X_{n}$ for $n\in\Z\backslash\{0\}$.
We let
\bas{
\tilde{\bX}(n)=(\tilde X_{n+k})_{k\in-\N_{0}}.
}
Thus $\tilde{\bX}(n)$ agrees with  $\bX(n)$ in all but at most one component.

We let $(\cF_{n})_{n\in\Z}$ be the filtration generated by $(X_{n})$ and consider for fixed $N\ge n_{0}$ the martingale $(M_{n})_{n\in\Z}$ given by
\bas{
M_{n}=\E[\bar\Theta_{N}|\cF_{n}]-\E[\bar\Theta_N]
}
We denote by $([M]_{n})_{n\in\Z}$ its quadratic variation process being defined by
\bas{
[M]_{n}=\sum_{r=-\infty}^{n}(M_{r}-M_{r-1})^{2}
}
and analyse the $q/2$-th moment of its increments $\Delta [M]_{n}:=  (M_{n}-M_{n-1})^{2}$.
Note that for every $n\in\Z\cap (-\infty,N]$, one has that
\bas{
\Delta M_{n}=\E[\bar\Theta_N|\cF_{n}]-\E[\bar\Theta_N|\cF_{n-1}] = \E\Bigl[\sum_{r=n_{0}+1}^{N} \gamma_{r} g(\mathbf X(r))|\cF_{n}\Bigr]-  \E\Bigl[\sum_{r=n_{0}+1}^{N} \gamma_{r}g(\mathbf X(r))|\cF_{n-1}\Bigr].
}
Now using that $(X_{n})$ is an i.i.d.\ sequence we conclude that, for almost all $\omega\in\Omega$,
\bas{
\E\Bigl[\sum_{r=n_{0}+1}^{N} \gamma_{r} g(\mathbf X(r))\Big|\cF_{n}\Bigr](\omega)=\E\Bigl[\sum_{r=n_{0}+1}^{N} \gamma_{r} g(\bX(r-n))\Big| \forall k\in -\N_{0} : X_{k}=X_{k+n}(\omega) \Bigr]
}
and
\bas{
\E\Bigl[\sum_{r=n_{0}+1}^{N} \gamma_{r} g(\mathbf X(r))\Big|\cF_{n-1}\Bigr](\omega) &=
\E\Bigl[\sum_{r=n_{0}+1}^{N} \gamma_{r} g(\bX(r-n))\Big| \forall k\in -\N : X_{k}=X_{k+n}(\omega) \Bigr]\\
&=\E\Bigl[\sum_{r=n_{0}+1}^{N} \gamma_{r} g(\tilde {\bX}(r-n))\Big| \forall k\in -\N_{0} : X_{k}=X_{k+n}(\omega) \Bigr].
}
Consequently,
\bas{
\E\Bigl[&\sum_{r=n_{0}+1}^{N} \gamma_{r} g(\mathbf X(r))\Big|\cF_{n}\Bigr](\omega)-
\E\Bigl[\sum_{r=n_{0}+1}^{N} \gamma_{r}g(\mathbf X(r))\Big|\cF_{n-1}\Bigr](\omega)\\
&=\E\Bigl[\underbrace{\sum_{r=n_{0}+1}^{N} \gamma_{r}(g(\mathbf X(r-n))-g(\tilde {\mathbf X}(r-n))) }_{=:\Upsilon_{n}[N]} \Big| \forall k\in -\N_{0} : X_{k}=X_{k+n}(\omega) \Bigr].
}
By stationarity, the distribution of the conditional expectation 
\bas{\omega\mapsto \E\bigl[\Upsilon_{n}[N]\big|\forall k\in -\N_{0} : X_{k}=X_{k+n}(\omega)\bigr]
}
is the same as the distribution of $\E\bigl[\Upsilon _n[N]\big| X_{0},X_{-1},\dots \bigr]$. 
Using this and Jensen's inequality we get that
\bas{
\E\bigl[\Delta [M]_{n}^{p/2}\bigr]= \E\bigl[|M_{n}-M_{n-1}|^{p}\bigr]=\E\bigl[\bigl |\E\bigl[\Upsilon _n[N]\big|X_{0},X_{-1},\dots\bigr]\bigr|^{p}\bigr] 
\le \E\bigl[|\Upsilon_n[N]|^{p}\bigr].
}
It follows with the triangle inequality that
\bas{\label{eq:723764}
\E\bigl[[M]_{N}^{p/2}\bigr]^{2/p}=\E\Bigl[\Bigl(\sum_{n=-\infty}^{N} \Delta [M]_{n}\Bigr)^{p/2}\Bigr]^{2/p}\le \sum_{n=-\infty}^{N} \E\bigl[(\Delta [M]_{n})^{p/2}\bigr]^{2/p}\le \sum_{n=-\infty}^{N}
\E\bigl[|\Upsilon_n[N]|^{p}\bigr]^{2/p}.
}
Now for $n\in\Z$ with $n\le N$ we get with Lemma~\ref{le:23456} that
\bas{
|\Upsilon_{n}[N]|\le\sum_{r=n\vee (n_{0}+1)}^{N} \gamma_{r} |g(\mathbf X(r-n))-g(\tilde {\mathbf X}(r-n))|\le \sum_{r=n\vee (n_{0}+1)}^{N} \gamma_{r} \varrho_{n-r} |X_{0}-\tilde X_{0}|
}
and again by the triangle inequality we have in terms of $C:=\E[|X_{0}-\E[ X_{0}]|^{p}]^{1/p}$,
\bas{
\E[|\Upsilon_{n}[N]|^{p}]^{1/p} &\le \sum_{r=n\vee (n_{0}+1)}^{N} \gamma_{r} \varrho_{n-r} \,\E[|X_{0}-\tilde X_{0}|^{p}]^{1/p}\\
&\le\begin{cases} 2C \|\varrho\|_{\ell_1} \gamma_{n} ,&\text{ if } n\ge n_{0}+1,\\
2C(1-\sqrt\beta)^{-1}\varrho_{n-(n_{0}+1)}\gamma_{n_{0}+1} ,&\text{ if } n\le n_0.
\end{cases}
}
Together with~(\ref{eq:723764}) we conclude that
\bas{
\E\bigl[[M]_{N}^{p/2}\bigr]^{2/p}&\le\sum_{n=-\infty}^{N}
\E[|\Upsilon_n[N]|^{p}]^{2/p}\\
& \le 4C^{2} (1-\sqrt\beta)^{-2}\gamma_{n_{0}+1}^{2}\sum_{n=-\infty}^{n_{0}} \varrho_{n-(n_{0}+1)}^{2}+4C^{2} \|\varrho\|_{\ell_{1}} ^{2}\sum_{n=n_{0}+1}^{N}
\gamma_{n}^{2}\\
&\le 4C^{2}\Bigl((1-\sqrt\beta)^{-2} \varrho_{0} \|\varrho\|_{\ell_{1}} \gamma_{n_{0}+1}^{2}+ \|\varrho\|_{\ell_{1}} ^{2}\sum_{n=n_{0}+1}^{N}
\gamma_{n}^{2}\Bigr).
}
The statement follows with the Burkholder-Davis-Gundy inequality since $\bar\Theta_{N}=\E[\bar\Theta_{N}]+M_{N}$.
\end{proof}

Next, we derive an extension of the previous lemma that considers whole trajectories instead of single time instances.

\begin{lemma} \label{le:1345}   
Let $p\in(2,\infty)$. One has for every  $n\in\N_{0}$ with $n\ge n_{0}$
\bas{
\E\Bigl[\max_{k=n_0,\dots,n} |\bar\Theta_k-\E[\bar\Theta_k]|^{p}\Bigr]^{1/p}\le  (1-2^{-(\frac12-\frac 1p)})^{-1} \kappa \gamma_{n_{0}+1}  \sqrt {n-n_{0}}\, \E\bigl[|X_{0}-\E[X_{0}]|^{p}\bigr]^{1/p} ,
}
where $\kappa$ is as in~(\ref{eq:72635984}).
\end{lemma}

\begin{proof}
Without loss of generality we can assume that $n_{0}=0$ since the general statement then follows by an index shift.

Suppose that $m\in\N$ and $2^{m-1}\le N\le 2^{m}-1$.
We call for every $k=1,\dots,m$, all intervals of the form
\bas{
[j,j+2^{m-k})\qquad (j\in (2^{m-k+1}\N_{0})\cap \{0,\dots,N\})
}
and additionally the empty set an interval of the $k$-th layer. Note that by construction for every natural number  $n\in\{0,\dots,N\}$ the interval $[0,n)$ can be uniquely written as pairwise disjoint union
\bas{
[0,n)=[0,n_{1})\cup [n_{1},n_2)\cup \ldots\cup [n_{m-1},n_m)
}
with each  $[n_{k-1},n_{k})$ being an interval of the $k$-th layer. Indeed, the representation is obtained for $n_{k}=2^{m-k}\lfloor 2^{-(m-k)} n\rfloor$.
We let $\bar \Theta_{k,\ell}:=\bar\Theta_{\ell}-\bar\Theta_{k}-\E[\bar \Theta_{\ell}-\bar \Theta_{k}]$ for every $k,\ell\in\N_{0}$ with $0\le k\le \ell$. Then for~$n$ as above we have
\bas{
 \bar \Theta_{0,n}= \sum_{k=1}^{m}  \bar \Theta_{n_{k-1},n_{k}}.
}
Consequently,
\bas{
\max_{n\in\{0,\dots ,N\}} |\bar\Theta_{0,n}|\le \sum_{k=1}^{m} \max_{[n_{k-1},n_{k}):\text{ interval of $k$th layer}} |\bar\Theta_{n_{k-1},n_{k}}|
}
and
\begin{align}\label{eq8362}
\E\Bigl[\max_{n\in\{0,\dots ,N\}} |\bar\Theta_{0,n}|^{p}\Bigr]^{1/p} \le \sum_{k=1}^{m} \E\Bigl[\max_{[n_{k-1},n_{k}):\text{ interval of $k$th layer}} |\bar\Theta_{n_{k-1},n_{k}}|^{p}\Bigr]^{1/p}.
\end{align}
Now note that by Lemma~\ref{le:842} one has in terms of $C:=\E[|X_{0}-\E[ X_{0}]|^{p}]^{1/p}$ that for an arbitrary non-empty interval $[n_{k-1},n_k)$ of the $k$-th layer
\bas{
\E[|\bar\Theta_{n_{k-1}:n_k}|^p]^{2/p}\le \kappa^2  \gamma_1^2 2^{m-k} C^2
}
so that
\bas{
\E\Bigl[\max_{[n_{k-1},n_{k}):\text{ interval of $k$th layer}} |\bar\Theta_{n_{k-1},n_{k}}|^{p}\Bigr]&\le \sum_{[n_{k-1},n_{k}):\text{ interval of $k$th layer}} \E[ |\bar\Theta_{n_{k-1},n_k } |^p]\\
&\le  \kappa^{p} \gamma_{1}^{p}  C^{p} 2^{k-1} (2^{m-k})^{p/2}.
}
Together with~(\ref{eq8362}) we obtain that
\bas{
\E\Bigl[\max_{n\in\{0,\dots ,N\}} |\bar\Theta_{0,n}|^{p}\Bigr]^{1/p} \le \kappa \gamma_{1}  C \sqrt{2^{m-1}}  \sum_{k=1}^{m} 2^{-(\frac12-\frac 1p)(k-1)}\le (1-2^{-(\frac12-\frac 1p)})^{-1} \kappa \gamma_{1}  C \sqrt N.
}
%
\end{proof}
\black

\section{Introducing approximations of the Adam optimization processes}
\label{sec:approx_intro}

In the proof of Theorem~\ref{thm-2}, we will work with approximations to the  Adam algorithm that are sufficiently  close but easier to analyse. 
For this we fix damping parameters $ \alpha, \beta, \epsilon \in [0,\infty) $ 
with $ \alpha < \sqrt\beta < 1 $ and an innovation $(X,U)$ and denote by 
$
  ( \theta_n )_{ n \ge n_0 } 
$ 
an Adam algorithm started at time $ n_0 \in \N_0 $ 
in a state $ ( \theta_{ n_0 }, m_{ n_0 }, v_{ n_0 } ) $. 
We assume that there exists a deterministic 
$
  \mathbf x=(x_k)_{ k \in - \N_0 }\in\ell_{ \varrho }^d
$ 
such that for all $ i \in \{ 1, \dots, d \} $ we have 
\bas{
m_{n_{0}}=(1-\alpha) \sum_{k\in-\N_{0}} \alpha^{-k} x_{k} 
\qquad\text{and} 
\qquad 
 v^{(i)}_{n_{0}}=(1-\beta) \sum_{k\in-\N_{0}} \beta^{-k} ( x^{(i)}_k )^2
   .
}
We consider the $\R^{d}$-valued process $ ( X_k )_{ k \in \Z } $ given by
\bas{
X_{k}=\begin{cases} x_{k-n_{0}}, & \text{ if } k\le n_{0},\\
X(U_{k},\theta_{k-1}), &\text{ if } k> n_{0} \end{cases}
}
and we set for every $n\ge n_{0}$, $\bX(n)=(X_{n+k})_{k\in-\N_{0}}$.


The terms $m_n$ and $v_n$ (defined as in Def.~\ref{def:Adam}) may be directly represented as weighted series over the past innovations  and we get that for every $n\in\N$ and $i=1,\dots,d$ one has that
\bas{
\sigmai_{n}m_{n}^{(i)}= \frac{(1-\alpha) \sum_{k=-\infty}^{0} \alpha^{-k} X^{(i)}_{n+k}}{
  \epsilon + \sqrt {(1-\beta)/(1-\beta^{n})\, \sum_{k=-\infty}^{0} \beta^{-k}(X^{(i)}_{n+k})^{2} } }.
}
Note that up to an additional factor $(1-\beta^{n})$ the latter expression agrees with 
$g(\bX^{(i)}(n))$ (cf.\ \cref{eq872346}). 
This means that for large $n$ the increments of the process $(\theta_{n})$ are similar to the ones of the process $(\Theta_{n})_{n\ge n_{0}}$ given by
\bas{\label{eq:3458741}
\Theta_{n}:=\sum_{k=n_{0}+1}^{n} \gamma_{k} \,g(\bX(k)).
}

For technical reasons we will consider dynamical systems with 
simpler dynamics. For this denote by  $(n_{\ell})_{\ell\in\N_{0}}$ an increasing  $\N_{0}$-valued sequence tending to infinity.  Here deliberately the initial value $n_0$ of the series coincides with the initial time of the Adam scheme.
We will use the times in the sequence $(n_{\ell})$ as update-times for the $\theta$-parameter similar to an Euler approximation. More explicitly, we define 
for every $\ell\in\N$ and 
$n\in\{n_{\ell-1}+1,\dots,n_{\ell}\}$, 
the process $\tilde \bX(n)=(\tilde \bX(n)_{k})_{k\in-\N_{0}}$ by
\bas{
\tilde \bX(n)_{k}=\begin{cases}
X(U_{n+k}, \theta_{n_{\ell-1}}), & \text{ \ if } n+k>n_{\ell-1}, \\
X_{n+k}, &\text{ \ else.}
\end{cases}
}
Comparing $\bX(n)$ and $\tilde\bX(n)$ we see that the terms outside the $(n_{\ell-1},n_\ell]$-window agree and inside the window the $\theta$-parameter is fixed as $\theta_{n_{\ell-1}}$. 
In analogy to before  we consider the process $(\tilde\Theta_{n})_{n\ge n_{0}}$ 
given by 
\bas{
\label{eq:approximation_process_1}
\tilde \Theta_{n}:=\sum_{k=n_{0}+1}^{n} \gamma_k\, g(\tilde \bX(k)).
}

Next, let $(U_{\ell,n})_{\ell\in\N,n\in\Z}$ be a family of independent copies of $U$ that is also independent of $(U_{n})_{n\in\N}$. For every $\ell\in\N$ and $n\in \{n_{\ell-1}+1,\dots,n_{\ell}\}$, we consider $\dtilde \bX(n)=(\dtilde \bX(n)_{k})_{k\in-\N_{0}}$ given by
\bas{
\dtilde \bX(n)_{k}=\begin{cases}
X(U_{n+k}, \theta_{n_{\ell-1}})=\tilde \bX(n)_{k}, & \text{ \ if } n+k>n_{\ell-1}, \\
X(U_{\ell,n+k}, \theta_{n_{\ell-1}}), &\text{ \ else.}
\end{cases}
}
In analogy to before we define a process $(\dtilde \Theta_{n})_{n\ge n_{0}}$ via
\bas{
\label{eq:approximation_process_2}
\dtilde \Theta_{n}=\sum_{k=n_{0}+1}^{n} \gamma_{k}\, g(\dtilde \bX (k)).
}
We see that in $\dtilde \bX(n)$ and $\tilde \bX(n)$ the terms inside the $(n_{\ell-1},n_\ell]$-window agree. Moreover, the terms are chosen in such a way that conditionally  on $\theta_{n_{\ell-1}}$, $\dtilde \bX(n)$ is a sequence of i.i.d.\ random variables. Roughly speaking, this entail that the related approximation is in ``stationary equilibrium''.

As last approximation we denote by $(\bar \Theta_{n})_{n\ge n_{0}}$ the unique process satisfying for every $\ell\in\N$ and $n=n_{\ell-1}+1,\dots,n_{\ell}$
\bas{
\Delta\bar\Theta_{n}:=\bar\Theta_{n}-\bar\Theta_{n-1}= \gamma_{n} f(\theta_{n_{\ell-1}})
}
and $\bar\Theta_{n_{0}}=0$.

%
%
%

In the analysis, we will impose the following regularity assumption on the innovation.

\begin{defi}\label{def:V2}
Let $C,\tilde C,\tilde L\in[0,\infty)$, $p\in[2,\infty)$ and $V\subset\R^{d}$ a measurable set. An innovation $(X,U)$ is  called \emph{$p$-regular  with parameter $(C,\tilde C,\tilde L)$ on $V$}, if for every $\theta,\theta'\in V$, one has
\bas{ \E[|X(U,\theta)|^{2}]^{1/2}\le C,  \ & \E[|X(U,\theta)-\E[X(U,\theta)]|^{p}]^{1/p}\le \tilde C \text{ \ and}\\ 
&\qquad \E[|X(U,\theta)-X(U,\theta')|^{p}]^{1/p} \le \tilde L  \,|\theta-\theta'|. 
}
\end{defi}

We give a quantitative statement that allows us to control the error between the Adam algorithm and its approximations.

\begin{prop}\label{prop:234}
Let
\begin{enumerate}[label=(\roman*)]
\item $\alpha,\beta\in[0,1)$ with $\alpha<\sqrt \beta$ and $\epsilon\in(0,\infty)$ (the \emph{damping parameters}),
\item $n_0\in\N_{0}$, $\theta_{n_{0}},m_{n_{0}}\in\R^{d}$, $v_{n_{0}}\in[0,\infty)^{d}$ (the initialisation), 
\item a decreasing $(0,\infty)$-valued sequence $(\gamma_{n})_{n\in\N}$(\emph{sequence of step-sizes}),
\item $(X,U)$ an innovation,
\item $C,\tilde C, \tilde L\in[0,\infty)$, $p\in[2,\infty)$ and $V\subset\R^{d}$ be a measurable set,
\item $(n_\ell)_{\ell\in\N_0}$ be a strictly increasing $\N_0$-valued sequence.
\end{enumerate}

Let $(X,U)$ be a $p$-regular innovation with parameter $(C,\tilde C,\tilde L)$ on $V$
and $(\theta_{n})_{n\in\N_0\cap[n_{0},\infty)}$ the Adam algorithm with damping parameter $(\alpha,\beta,\epsilon)$ and step-sizes $(\gamma_{n})$ started at time $n_{0}$ in $(\theta_{n_0},m_{n_{0}},v_{n_{0}})$.
Moreover, let
\bas{
\mathfrak N=\inf \{n\ge n_{0}:  \theta(n)\not \in V\}
}
and let $(t_n)_{n\in\N}$ as in~(\ref{def:tn}).
One has the following for the approximations $(\Theta_n)_{n\ge n_0}$, $(\tilde\Theta_n)_{n\ge n_0}$ and $(\dtilde\Theta_n)_{n\ge n_0}$:
\begin{enumerate}[label=(\Roman*)]
\item for every $\mathfrak n\in\{n_{0},n_{0}+1,\dots\}$,
\begin{align}\begin{split}
\E\Bigl[ \Bigl(\sum_{k=\mathfrak n+1}^{\infty} \1_{\{\mathfrak N\ge k\}} |\Delta \theta_{k}-\Delta \Theta_{k} |\Bigr)^{p}\Bigr]^{1/p}&\le \bigl(\kappa_{1} \beta^{\frac12(\mathfrak n+1)-n_{0}}\|\bx\|_{\ell^{d}_\varrho}
+\kappa_{2} C  \bigr) \gamma_{\mathfrak n+1} \beta^{\mathfrak n+1}
\end{split}\end{align}
\item for every $\ell\in\N$ and $\mathfrak n,n\in\N_{0}$ with $n_{\ell-1}\le \mathfrak n\le n\le n_{\ell}$,
\begin{align}\begin{split}
\E\Bigl[ \Bigl(\sum_{k=\mathfrak n+1}^n \1_{\{\mathfrak N\ge k\}} |\Delta \Theta_{k} -\Delta \tilde\Theta_{k}|\Bigr)^{p}\Bigr]^{1/p}&\le \kappa_{3}\tilde L  (t_{n_{\ell}}-t_{n_{\ell-1}})(t_{n}-t_{\mathfrak n})
\end{split}\end{align}
\item for every $\ell\in\N$,
\begin{align}
\label{eq78246-2}
\E\Bigl[&\1_{\{\mathfrak N>n_{\ell-1}\}}\Bigl(\sum _{k=n_{\ell-1}+1}^{n_{\ell}} \bigl|\Delta {\tilde \Theta}_{k}-\Delta \dtilde\Theta_{k}\bigr|\Bigr)^{p}\Bigr]^{1/p}\le \gamma_{n_{\ell-1}+1} \bigl(\kappa_{4} C+\kappa_{5}\beta^{(n_{\ell-1}-n_{0})/2}\|\bx\|_{\ell_{\varrho}^{d}}),
\end{align}
\item[(IV.a)]for every $\ell\in\N$,
\bas{
\E[\1_{\{\mathfrak N>n_{\ell-1}\}}|\dtilde\Theta_{n_{\ell}}-\dtilde \Theta_{n_{\ell-1}}-(t_{n_{\ell}}-t_{n_{\ell-1}})f(\theta_{n_{\ell-1}})|^{p}|\cF_{n_{\ell-1}}]^{1/p} \le \kappa_{6} \tilde C\Bigl(\sum_{k=n_{\ell-1}+1}^{n_{\ell}} \gamma_k^{2}\Bigr)^{1/2},
}
\item[(IV.b)]if $p>2$, then for every $\ell\in\N$,
\begin{align}
\begin{split}\label{eq78246-3}
\E[\1_{\{\mathfrak N>n_{\ell-1}\}}\max_{n=n_{\ell-1},\dots,n_{\ell}} |\dtilde\Theta_{n}&-\dtilde \Theta_{n_{\ell-1}}-(t_{n}-t_{n_{\ell-1}})f(\theta_{n_{\ell-1}})|^{p}|\cF_{n_{\ell-1}}]^{1/p}\\
& \le (1-2^{-(\frac12-\frac 1p)})^{-1}  \kappa_{6} \tilde C \gamma_{n_{\ell-1}+1} \sqrt {n_{\ell}-n_{\ell-1}},
\end{split}\end{align}
\end{enumerate}
where 
\bas{ \label{eq237643}
&\kappa_{1}=\frac2{1-\beta^{3/2}}, \ \kappa_{2}=\frac2{1-\beta^{3/2}}\bigl(\frac{ \varrho_{0}}{1-\beta} +\|\varrho\|_{\ell_{1}}\bigr), \kappa_{3}=\frac{1-\alpha}{\sqrt{1-\beta}\sqrt{1-\alpha^{2}/\beta}} d \|\varrho\|_{\ell_{1}},\\ 
&\kappa_{4}=2(1-\sqrt\beta)^{-1}\|\varrho\|_{\ell_{1}}, \  
\kappa_{5}=(1-\sqrt\beta)^{-1}, \ \kappa_{6}=2  C_{p}\bigl((1-\sqrt\beta)^{-1} \sqrt{\varrho_{0} \|\varrho\|_{\ell_{1}}}+ \|\varrho\|_{\ell_{1}} \bigr),}
$(\varrho_{k})$ is as in~(\ref{def:varrho}) and $C_{p}$ is the constant in the Burholder-Davis-Gundy inequality when applied for the $p$th moment.
 \end{prop}

\begin{proof}
1) We prove the first inequality.
Note that
\begin{align}\begin{split}\label{eq:235631}
\bigl|\sigma^{(i)}_k m_{k}^{(i)}-g(\mathbf X^{(i)}(k))\bigr|&\le \Bigl(1-\frac1{\sqrt{1/(1-\beta^{k})}}\Bigr) |g(\mathbf X^{(i)}(k))|\le 2\beta^{k} \|\bX^{(i)}(k)\|_{\ell_\varrho}\\
&\le 2 \beta^{k}\Bigl({\sqrt\beta}^{k-n_{0}}\|\bx^{(i)}\|_{\ell_\varrho}+\sum_{r=n_{0}+1}^{k} \varrho_{r-k} |X^{(i)}_r|\Bigr).
\end{split}\end{align}
This entails that
\bas{
\E[\1_{\{\mathfrak N\ge k\}} |\Delta\theta_{k}-\Delta\Theta_{k}|^{p}]^{1/p}\le 2 \beta^{k}\Bigl({\sqrt\beta}^{k-n_{0}}\|\bx\|_{\ell^{d}_\varrho}+C\sum_{r=n_{0}+1}^{k} \varrho_{r-k}\Bigr).
}
Consequently,
\bas{
\E\Bigl[\Bigl(\sum_{k=\mathfrak n+1}^{\infty} \1_{\{\mathfrak N\ge k\}}\gamma_{k}& |\Delta \theta _k -\Delta \Theta_{k}|\Bigr)^{p}\Bigr]^{1/p}
\le 2\sum_{k=\mathfrak n+1}^{\infty} \gamma_{k}  \beta^{k}\Bigl({\sqrt\beta}^{k-n_{0}}\|\bx\|_{\ell^{d}_\varrho}+C\sum_{r=n_{0}+1}^{k} \varrho_{r-k} \Bigr)\\
&\le 2(1-\beta^{3/2})^{-1}\gamma_{\mathfrak n+1} \beta^{\frac32(\mathfrak n+1)-n_{0}}\|\bx\|_{\ell^{d}_\varrho}+2C\gamma_{\mathfrak n+1}\sum_{r=n_{0}+1}^{\infty} \sum_{k=r\vee (\mathfrak n+1)}^{n}\beta^{k} \varrho_{r-k} 
}
In the case where $r\le \mathfrak n$, one has that
\bas{
\sum_{k=r\vee (\mathfrak n+1)}^{n}\beta^{k} \varrho_{r-k} \le (1-\beta^{3/2})^{-1}\varrho_{r-(\mathfrak n+1)} \beta^{\mathfrak n+1}
}
and in the case where $r>\mathfrak n$,
\bas{
\sum_{k=r\vee (\mathfrak n+1)}^{n}\beta^{k} \varrho_{r-k} \le (1-\beta^{3/2})^{-1}\varrho_{0}  \beta^{r}.
}
Hence,
\bas{
\E\Bigl[&\Bigl(\sum_{k=\mathfrak n+1}^{\infty} \1_{\{\mathfrak N\ge k\}}\gamma_{k} |\Delta \theta _k -\Delta \Theta_{k}|\Bigr)^{p}\Bigr]^{1/p}\\
&\le  2(1-\beta^{3/2})^{-1}\gamma_{\mathfrak n+1} 
\Bigl( \beta^{\frac32(\mathfrak n+1)-n_{0}}\|\bx\|_{\ell^{d}_\varrho}+\beta^{\mathfrak n+1}C\sum_{r=n_{0}+1}^{\mathfrak n} \varrho_{r-(\mathfrak n+1)} 
 +C \sum_{r=\mathfrak n+1}^{n}\varrho_{0} \beta^{r} \Bigr)\\
 &\le 2(1-\beta^{3/2})^{-1}\bigl( \beta^{\frac12(\mathfrak n+1)-n_{0}}\|\bx\|_{\ell^{d}_\varrho}
+ (\|\varrho\|_{\ell_{1}}+(1-\beta)^{-1} \varrho_{0} ) C  \bigr) \gamma_{\mathfrak n+1} \beta^{\mathfrak n+1}.
}

2.) Fix $\ell\in\N$. Since for every $k\ge n_{0}$
\bas{
\sigma^{(i)}_{k} m^{(i)}_{k} \le \frac{1-\alpha}{\sqrt{1-\beta}}\frac1{\sqrt{1-\alpha^{2}/\beta}},
}
one has  for every $n=n_{\ell-1}+1,\dots,n_{\ell}$
\bas{
|\theta_{n}^{(i)}-\theta_{n_{\ell-1}}^{(i)}|=\Bigl |\sum_{k=n_{\ell-1}+1}^{n}\gamma_{k} \,\sigma^{(i)}_{k} m^{(i)}_{k} \Bigr|\le \frac{1-\alpha}{\sqrt{1-\beta}}\frac1{\sqrt{1-\alpha^{2}/\beta}}(t_{n}-t_{n_{\ell-1}}).
}
This implies that
\bas{
\E[\1_{\{\mathfrak N\ge n\}}|X(U_{n},\theta_{n-1})-X(U_{n},\theta_{n_{\ell-1}})|^{p}|\cF_{n-1}]^{1/p} \le \underbrace{\frac{1-\alpha}{\sqrt{1-\beta}}\frac1{\sqrt{1-\alpha^{2}/\beta}} d}_{=\kappa_{3}/\|\varrho\|_{\ell_{1}}} \tilde L (t_{n_{\ell}}-t_{n_{\ell-1}}) .
}
Now note that for $k=n_{\ell-1}+1,\ldots,n_{\ell}$
\bas{
\bigl| g(\mathbf X(k))-g(\tilde\bX (k))\bigr|\le \sum_{r=n_{\ell-1}+1}^{k} \rho_{r-k}\, |X(U_{r},\theta_{r-1})-X(U_{r},\theta_{n_{\ell-1}})|.
}
This implies that for $\mathfrak n,n\in\N_{0}$ with $n_{\ell-1}\le \mathfrak n\le n\le n_{\ell}$
\bas{
\E\Bigl[\Bigl(&\sum_{k=\mathfrak n+1}^{n} \1_{\{\mathfrak N>k\}}\gamma_{k}\bigl| g(\mathbf X(k))-g(\tilde\bX (k))\bigr|\Bigr)^{p}\Bigr]^{1/p}\\
&\le\sum_{k=\mathfrak n+1}^{n}\gamma_{k} \,\E\Bigl[\Bigl( \1_{\{\mathfrak N>k\}}\bigl| g(\mathbf X(k))-g(\tilde\bX(k))\bigr|\Bigr)^{p}\Bigr]^{1/p}\\
&\le \sum_{k=\mathfrak n+1}^{n}\gamma_{k} \,\E\Bigl[\Bigl( \1_{\{\mathfrak N>k\}} \sum_{r=n_{\ell-1}+1}^{k}\rho_{r-k}\, |X(U_{r},\theta_{r-1})-X(U_{r},\theta_{n_{\ell-1}})|\Bigr)^{p}\Bigr]^{1/p}\\
&\le \sum_{k=\mathfrak n+1}^{n}\gamma_{k} \sum_{r=n_{\ell-1}+1}^{k} \varrho_{r-k} \,\E\Bigl[\Bigl( \1_{\{\mathfrak N>r\}} |X(U_{r},\theta_{r-1})-X(U_{r},\theta_{n_{\ell-1}})|\Bigr)^{p}\Bigr]^{1/p}\\
&\le \kappa_{3}  \tilde L (t_{n_\ell}-t_{n_{\ell-1}}) \sum_{k=\mathfrak n+1}^{n}\gamma_{k} .
}

3.) Note that 
\bas{
|\Delta \dtilde \Theta_{\mathfrak n,n}-\Delta\tilde \Theta_{\mathfrak n,n}|=|\gamma_{n} \, (g(\mathbf {\dtilde X}(n))-g(\mathbf {\tilde X}(n)))|\le \gamma_{n} \, \|\mathbf {\dtilde X}(n)-\mathbf {\tilde X}(n)\|_{\ell_{\varrho}^{d}}
}
and for $n=n_{\ell-1}+1,\dots , n_{\ell}$,
\bas{
\|\mathbf {\dtilde X}(n)-\mathbf {\tilde X}(n)\|_{\ell_{\varrho}^{d}}&=\sum_{k=-\infty}^{n_{\ell-1}} \varrho_{k-n} |X_{k}-X(U_{\ell,k},\theta_{n_{\ell-1}})|\\
&\le \sum_{k=-\infty}^{n_{\ell-1}} \varrho_{k-n} |X_{k}|+\sum_{k=-\infty}^{n_{\ell-1}} \varrho_{k-n} |X(U_{\ell,k},\theta_{n_{\ell-1}})|.
}
Hence,
\bas{
\E\bigl[\1_{\{\mathfrak N>n_{\ell-1}\}} \|\mathbf {\dtilde X}(n)-\mathbf {\tilde X}(n)\|^{p}_{\ell_{\varrho}^{d}}\bigr]^{1/p}]&\le \|\bx\|_{\ell_{\varrho}^{d}} {\sqrt\beta}^{n-n_{0}}+2C\sum_{k=-\infty}^{n_{\ell-1}}\varrho_{k-n}\\
&\le  \|\bx\|_{\ell_{\varrho}^{d}} {\sqrt\beta}^{n-n_{0}}+2(1-\sqrt\beta)^{-1}C\varrho_{n_{\ell-1}-n} 
}
and
\bas{
\E\Bigl[\1_{\{\mathfrak N >n_{\ell-1}\}}&\sum _{n=n_{\ell-1}+1}^{n_{\ell}}|\Delta \dtilde \Theta_{n}-\Delta\tilde \Theta_{n}|^{p}\Bigr]^{1/p}\le \gamma_{n_{\ell-1}+1} \sum_{n=n_{\ell-1}+1}^{n_{\ell}}\bigl( \|\bx\|_{\ell_{\varrho}^{d}} {\sqrt\beta}^{n-n_{0}}+2(1-\sqrt\beta)^{-1}C\varrho_{n_{\ell-1}-n} \bigr)\\
&\le\frac{1}{1-\sqrt\beta}\gamma_{n_{\ell-1}+1} \bigl( {\sqrt\beta}^{n_{\ell-1}-n_{0}} \|\bx\|_{\ell_{\varrho}^{d}} +2\|\varrho\|_{\ell_{1}} C \bigr).
}

4.) The inequalities (IV.a) and (IV.b) are direct consequences of Lemmas~\ref{le:842} and~\ref{le:1345} when conditioning on $\cF_{n_{\ell-1}}$.
\end{proof}

When choosing in estimate (II), $\mathfrak n=n_{\ell-1}$ and $n=n_{\ell}$ we obtain an estimate covering the same range of indices as the estimate (III). In order to obtain good results we will choose the sequence $(n_{\ell})_{\ell\in\N_{0}}$ in such a way that both error terms are of the same order. This motivates the following definition.

\begin{defi}Let $n_{0}\in\N_{0}$ and $\rho\in[ \sqrt{\gamma_{n_{0}+1}},\infty)$. We call the $\N_{0}$-valued sequence $(n_{\ell})_{\ell\in\N_{0}}$ a \emph{$\rho$-partition w.r.t.\ $(\gamma_{n})$} if  for every $\ell\in\N$, $n_{\ell}$ is the largest integer with
\bas{
t_{n_{\ell}}- t_{n_{\ell-1}}\le \rho \sqrt {\gamma_{n_{\ell-1}+1}},
}
where $(t_n)$ is as in~(\ref{def:tn}).
We call $n_{0}$ the \emph{starting value}.
\end{defi}

\begin{rem}
In the definition we  assume that $\rho\ge \sqrt{\gamma_{n_{0}+1}}$. This  guarantees that for every $\ell\in\N$, one has $\rho \sqrt{\gamma_{n_{\ell-1}+1}}\ge \gamma_{n_{\ell-1}+1}$ by monotonicity of $(\gamma_{n})$ so that $n_{\ell}>n_{\ell-1}$. This implies that $(n_\ell)$ is strictly increasing and that $\lim_{n\to\infty}n_{\ell}=\infty$.
\end{rem}

\begin{lemma}\label{le:3571}
Let $(\gamma_{n})_{n\in\N}$ be a $(0,\infty)$-valued decreasing non-summable sequence 
and let
 $(n_{\ell})_{\ell\in\N_{0}}$ be a $\rho$-partition for $(\gamma_{n})$.
\begin{enumerate}[label=(\roman*)]
 \item For $\ell\in\N$ with $\gamma_{n_{\ell-1}+1}<\rho^{2}$, one has that
 \bas{
  \sqrt{\gamma_{n_{\ell-1}+1}}\le \frac1{\rho-\sqrt{\gamma_{n_{\ell-1}+1}}} (t_{n_{\ell}}-t_{n_{\ell-1}}).
 }
 \item
Assume that for all $n= n_{0}+1,n_{0}+2,\dots$
\bas{
\frac { \gamma_{n}- \gamma_{n+1}}{ \gamma_{n}^{2}}\le \zeta
}
and that $\rho\zeta\sqrt{\gamma_{n_{0}+1}}<1$. Set $K=(1-\zeta\rho\sqrt{\gamma_{n_{0}+1}})^{-1}$. One has for every $\ell\in\N$ that
\bas{
\frac{\gamma_{n_{\ell-1}+1}}{\gamma_{n_{\ell}+1}}\le K \text{ \ and \ } \frac{ \gamma_{n_{\ell-1}+1}}{ \gamma_{n_{\ell}+1}}\le 1+\zeta K (  t_{n_{\ell}}- t_{n_{\ell-1}}) .
}
\end{enumerate}
\end{lemma}

\begin{proof} (i): This is an immediate consequence of
\bas{
t_{n_{\ell}}-t_{n_{\ell-1}}\ge \rho \sqrt{\gamma_{n_{\ell-1}+1}}-\gamma_{n_{\ell-1}+1}=(\rho- \sqrt{\gamma_{n_{\ell-1}+1}})  \sqrt{\gamma_{n_{\ell-1}+1}}.
}
(ii): Denote by $f:[0, \infty)\to(0,\infty)$ the function satisfying 
\bas{
f( t_{n})=  \gamma_{n+1}
}
for every $n\in\N_{0}$ and that is piecewise linear in between the points of $\{t_0,t_1,\ldots\}$.
First note that on each interval $( t_{n-1},t_{n})$
\bas{
-f'(s)= \frac{ \gamma_{n}-\gamma_{n+1}}{\gamma_{n}}\le \zeta \gamma_{n}
}
so that, for $\ell\in\N$,
\bas{
\gamma_{n_{\ell}+1}\ge \gamma_{n_{\ell-1}+1}- \zeta \gamma_{n_{\ell-1}+1} \rho \sqrt { \gamma_{n_{\ell-1}+1}}\ge (1- \zeta  \rho \sqrt { \gamma_{n_{0}+1}})\gamma_{n_{\ell-1}+1}
}
and
\bas{
\frac{\gamma_{n_{\ell-1}+1}}{\gamma_{n_{\ell}+1}}\le K.
}
One has
\bas{
 \gamma_{n_{\ell-1}+1}- \gamma_{n_{\ell}+1}=f(  t_{n_{\ell-1}})-f(  t_{n_{\ell}})\le \zeta  \gamma_{n_{\ell-1}+1} ( t_{n_{\ell}}- t_{n_{\ell-1}})
}
and
\bas{
\frac{ \gamma_{n_{\ell-1}+1}- \gamma_{n_{\ell}+1}}{ \gamma_{n_{\ell}+1}}\le \zeta K ( t_{n_{\ell}}-  t_{n_{\ell-1}}).
}
\end{proof}

We provide another technical estimate.
\red
\begin{lemma}\label{le:234787}
Let $\|\cdot\|$ be a norm induced by a scalar product $\llangle\cdot,\cdot\rrangle$ on $\R^d$ and let $p\ge 2$. One has for every $x,y\in\R^{d}$ that
\bas{
\|x+y\|^{p}\le \|x\|^{p}+ p \|x\|^{p-2} \llangle x,y\rrangle + 
\sfrac 12 p(p-1) (\|x\|\vee \|x+y\|)^{p-2}\|y\|^{2}.
}
\end{lemma}

\begin{proof} Let $x,y\in\R^{d}$. Note that the function $h(z)=\|z\|^{p}$ $(z\in\R^{d})$ is twice continuously differentiable with differentials
\bas{
Dh(z)(a)=p\,\|z\|^{p-2} \llangle z,a\rrangle \text{ \ and \ }D^{2}h(z)(a,b)=p(p-2) \|z\|^{p-4} \llangle z,a\rrangle \llangle z,b\rrangle+ p |z|^{p-2}\llangle a,b \rrangle.
}
By Taylor's theorem, we thus have
\bas{|x+y|^{p}=|x|^{p}+p\,|x|^{p-2} \langle x, y\rangle+ \int_{0}^{1} D^{2} h(x+ry)(y,y)(1-r)\,\dd r.
}
Note that for every $r\in[0,1]$,  one has $|x+ry|\le |x|\vee|x+y|$ and
\bas{
\|D^{2} h(x+ry)(y,y)\|&\le p(p-2) \|x+ry\|^{p-4} \|\llangle x+ry,y\rrangle \llangle x+ry,y\rrangle|+ p \|x+ry\|^{p-2}\llangle y,y \rrangle\\
&\le p(p-1)(\|x\|\vee\|x+y\|)^{p-2} \|y\|^{2}.
}
The statement follows since $\int_{0}^{1}(1-r)\,\dd r=\frac 12$.
\end{proof}
\black
%
%
%

\section{Ordinary differential equation (ODE) based error analyses for the Adam optimizer}
\label{sec:technical_estimates}

In this section we prove the central estimates for the analysis of the Adam algorithm. The analysis is done in such a way that the impact of certain perturbation effects is made completely transparent in the propositions. Although this precise information is not used when deriving the main theorems, the results allow to immediately investigate the impact of the tuning parameters of  Adam algorithms such as the damping parameters and the size of the mini-batch in the case where the innovation is obtained by the simulation of mini-batches.
So it is straight-forward to deduce variants of our main theorems, where the precise role of certain tuning parameters is revealed.

\begin{prop}[\ODE\ based non-uniform error analysis for the Adam optimizer] 
\label{prop35-12} Let $d\in\N$ and let $\|\cdot\|$ be a norm on $\R^{d}$ induced by a scalar product $\llangle\cdot,\cdot\rrangle$ satifying $\|x\|\le |x|$ for all $x\in\R^{d}$.

\begin{enumerate}[label=(\roman*)]
\item  
\label{item:1_non_uniform}
\emph{Innovation.} Let $C,\tilde C,C',L,\tilde L\in[0,\infty)$, $p\in[2,\infty)$ and suppose that $(X,U)$ is a  $p$-regular innovation with parameter $(C,\tilde C,\tilde L)$ on the set $V\subset\R^{d}$ and that the Adam vector field $f$ is bounded by $C'$ and $L$-Lipschitz on $V$ with respect to the norm $\|\cdot\|$.

\item 
\label{item:2_non_uniform}
\emph{Adam algorithm.} Let $(\theta_{n})_{n\ge n_{0}}$ be an Adam algorithm started at time $n_{0}\in\N_{0}$ in a state $(\theta_{n_0},m_{n_{0}},v_{n_{0}})$ and let $\bx\in\ell_{\varrho}^{d}$ with 
\bas{\label{eq34284}
m_{n_{0}}=(1-\alpha) \sum_{k\in-\N_{0}} \alpha^{-k} x_{k} \text{ \ and \ } v^{(i)}_{n_{0}}=(1-\beta) \sum_{k\in-\N_{0}} \beta^{-k} \bigl(x^{(i)}_k\bigr)^{2}
}
for all $i=1,\dots,d$.

\item 
\label{item:3_non_uniform}
\emph{ODE.} 
Let $\Psi:[t_{n_{0}},\infty)\to\R^{d}$ be a solution to the ODE
\bas{
\dot \Psi_{t} = f(\Psi_t)
}
staying in $V\subset\R^{d}$.

\item 
\label{item:4_non_uniform}
\emph{Local monotonicity of ODE.} 
Let $(\mathfrak R_{t})_{t\ge t_{n_{0}}}$ a decreasing $(0,\infty]$-valued mapping and suppose that for a fixed  $c_{1}\in(0,\infty)$  for every $t\in[t_{n_{0}},\infty)$
and $x\in V\cap \overline{B_{\|\cdot\|}(\Psi_{t},\mathfrak R_{t})}$ one has
\bas{
\llangle f(x)-f(\Psi_{t}),x-\Psi_{t}\rrangle\le -c_{1} \|x-\Psi_{t}\|^{2}.
}
\item 
\label{item:5_non_uniform}
\emph{$\rho$-partition.} Let $\rho\in[\sqrt{\gamma_{n_{0}+1}},\infty)$ and let $(n_{\ell})_{\ell\in\N_{0}}$ be a $\rho$-partition.
\item 
\label{item:6_non_uniform}
\emph{Technical assumptions.} Let $\kappa_{1},\dots,\kappa_{6}$ as in~(\ref{eq237643}) and $c',\zeta,\delta_{1},\delta_{2}\in(0,\infty)$. We suppose validity of the following inequalities:
\begin{enumerate}
\item  
\label{item:6a_non_uniform}
$\zeta\rho\sqrt{\gamma_{n_{0}+1}}<1$ and $\gamma_{n_{0}+1}\le \delta_{1}<\rho^{2}$
 
\item 
\label{item:6b_non_uniform}
for every $n= n_{0}+1,n_{0}+2,\dots$, one has
$
\frac { \gamma_{n}- \gamma_{n+1}}{ \gamma_{n}^{2}}\le \zeta
$
\item  
\label{item:6c_non_uniform}
$2c_{1}-L^2\rho\sqrt{\gamma_{n_{0}+1}} -\zeta K \ge 2c'$, for   $K:=(1-\zeta\rho\sqrt{\gamma_{n_{0}+1}})^{-1}$
\item   
\label{item:6d_non_uniform}
$\displaystyle{\frac\rho{\rho-\sqrt{\delta_{1}}}(\kappa_{3}\rho^{2}\tilde L+\sfrac12LC'\rho^{2}+(\kappa_{2}+\kappa_{4})C+(\kappa_1+\kappa_{5}) \|\bx\|_{\ell_{\varrho}^{d}})\sqrt{\gamma_{n_{0}+1}} \le \delta_{2}}$.
\end{enumerate}
\end{enumerate}
We consider the stopping time 
\bas{
\mathfrak N=\inf\{n\ge n_{0}: \theta_{n}\not\in V\}\wedge \inf\{n_{\ell}: \ell\in\N_{0}, \|\theta_{n_{\ell}}-\Psi_{t_{n_{\ell}}}\|>\mathfrak R_{t_{n_{\ell}}}\}
}
Then the sequence $(e_{\ell})_{\ell\in\N_{0}}$ given by
\bas{
e_{\ell}=\gamma_{n_{\ell}+1}^{-1} \,\E[\1_{\{\mathfrak N\ge n_{\ell}\} }\|\theta_{n_{\ell}}-\Psi_{t_{n_{\ell}}}\|^{p} ]^{2/p}
}
satisfies
\bas{\label{eq24811-0}
e_{\ell}\le (1-2c'(t_{n_{\ell}}-t_{n_{\ell-1}}) )\, e_{\ell-1} +\bigl(K(2 \sqrt{e_{\ell-1}} +\delta_{2}) a_{\ell}+b_{\ell}\bigr)(t_{n_{\ell}}-t_{n_{\ell-1}}),
}
where
\bas{\label{eq24811-1}
a_{\ell}= \frac1{\rho-\sqrt{\delta_{1}}} \aleph_{\ell}+\sfrac12 LC'\rho, \ \ 
b_{\ell}=\bigl((p-1)+(\sfrac12p(p-1))^{2/p}\bigr)\Bigl(2K\kappa_{6}^{2} \tilde C^{2} 
+\frac{4 K}{\rho-\sqrt{\delta_{1}}} \aleph_{\ell}^{2}\sqrt{\gamma_{n_{\ell-1}+1}}\Bigr)
}
and
\bas{\label{eq24811-2}
\aleph_{\ell}=\kappa_{3}\rho^{2}\tilde L  +(\kappa_{2}+\kappa_{4}) C+(\kappa_{1}+\kappa_{5})\beta^{(n_{\ell-1}-n_{0})/2}\|\bx\|_{\ell^{d}_{\varrho}}.
}
\end{prop}
\black

\begin{proof}
We use the approximations $(\Theta_n),(\tilde\Theta_n)$ and $(\dtilde \Theta_n)$ as introduced in Section~\ref{sec:approx_intro} and briefly write for $\mathfrak n,n\in\N_{0}$ with $n_{0}\le\mathfrak n\le n$,
\bas{
\Theta_{\mathfrak n:n}=\Theta_{n}-\Theta_{\mathfrak n}, \ \tilde\Theta_{\mathfrak n:n}=\tilde\Theta_{n}-\tilde\Theta_{\mathfrak n} \text{ \ and \ } \dtilde\Theta_{\mathfrak n:n}=\dtilde\Theta_{n}-\dtilde\Theta_{\mathfrak n}.
}
We derive a recursive estimate for $e_{\ell}$. For this we let
 \bas{ A_{\ell}=\1_{\{\mathfrak N>n_{\ell-1}\}} \E\bigl[ \theta_{n_{\ell}\wedge \mathfrak N}-\theta_{n_{\ell-1}}-\dtilde \Theta_{n_{\ell-1}:(n_{\ell}\wedge\mathfrak N)}\big|\cF_{n_{\ell-1}}\bigr]
}
and 
\bas{
M_{\ell}=\1_{\{\mathfrak N>n_{\ell-1}\}} \bigl(\dtilde \Theta_{n_{\ell-1}:n_{\ell}}-f(\theta_{n_{\ell-1}})(t_{n_{\ell}}-t_{n_{\ell-1}}) +  \theta_{n_{\ell}\wedge \mathfrak N}-\theta_{n_{\ell-1}}-\dtilde \Theta_{n_{\ell-1}:(n_{\ell}\wedge\mathfrak N)}\bigr)-A_{\ell}.
}
Note that on the event $\{\mathfrak N\ge n_{\ell}\}$ one has that
\bas{
\theta_{n_{\ell}}=\theta_{n_{\ell-1}}+ A_{\ell} +M_{\ell}
+f(\theta_{n_{\ell-1}})(t_{n_{\ell}}-t_{n_{\ell-1}}).
}
Moreover, write
\bas{
\Psi_{t_{n_{\ell}}}=\Psi_{t_{n_{\ell-1}}}+f(\Psi_{t_{n_{\ell-1}}}) (t_{n_{\ell}}-t_{n_{\ell-1}}) +\tilde A_{\ell},
}
where
$\tilde A_{\ell}= \Psi_{t_{n_{\ell}}}-(\Psi_{t_{n_{\ell-1}}}+f(\Psi_{t_{n_{\ell-1}}}) (t_{n_{\ell}}-t_{n_{\ell-1}}))$.
Before we derive appropriate estimates for $A_{\ell},\tilde A_{\ell},M_{\ell}$ we analyse how these estimates enter into the recursive estimates. Abridge 
\bas{
\Upsilon_{\ell}=\1_{\{\mathfrak N>n_{\ell-1}\}}(\theta_{n_{\ell-1}}-\Psi_{t_{n_{\ell-1}}}
+(f(\theta_{n_{\ell-1}})-f(\Psi_{t_{n_{\ell-1}}}))(t_{n_{\ell}}-t_{n_{\ell-1}})+
 A_{\ell} -\tilde A_{\ell} )
}
and note that $\Upsilon_{\ell}$  is $\cF_{n_{\ell-1}}$-measurable, $\E[M_{\ell}|\cF_{n_{\ell-1}}]=0$ and, on $\{\mathfrak N\ge n_{\ell}\}$, 
\bas{ \theta_{n_{\ell}}- \Psi_{t_{n_{\ell}}}= \Upsilon_{\ell}+M_{\ell}.
}
Using these properties together with  Lemma~\ref{le:234787} and the H\"older inequality  we get that
\begin{align}\begin{split}\label{eq:3256}
e_{\ell}&=\gamma_{n_{\ell}+1}^{-1}\,\E[\1_{\{\mathfrak N\ge n_{\ell}\} } \|\Upsilon_{\ell}+M_{\ell}\|^{p}   ]^{2/p}\le \gamma_{n_{\ell}+1}^{-1}\,\E[\|\Upsilon_{\ell}+M_{\ell}\|^{p}   ]^{2/p}\\
 &= \gamma_{n_{\ell}+1}^{-1}\,\E\bigl[\E[ \|\Upsilon_{\ell} +M_{\ell}\|^{p}|\cF_{n_{\ell-1}}]   \bigr]^{2/p}\\
 &\le \gamma_{n_{\ell}+1}^{-1} \,\E\bigl[ \|\Upsilon_{\ell} \|^{p}+ \sfrac 12p(p-1) (\|\Upsilon_{\ell}\|^{p-2}+\|M_{\ell}\|^{p-2})\|M_{\ell}\|^{2} \bigr]^{2/p}\\
  &= \gamma_{n_{\ell}+1}^{-1} \,\bigl(\E[ \|\Upsilon_{\ell} \|^{p}]+ \sfrac 12p(p-1) (\E[ \|\Upsilon_{\ell}\|^{p-2}|M_{\ell}|^{2}]+\E[\|M_{\ell}\|^{p}]) \bigr)^{2/p}\\
  &\le \gamma_{n_{\ell}+1}^{-1} \,\bigl(\E[ \|\Upsilon_{\ell} \|^{p}]+ \sfrac 12p(p-1) (\E[ \|\Upsilon_{\ell}\|^{p}]^{\frac{p-2}p}\,\E[\|M_{\ell}\|^{p}]^{2/p}+\E[\|M_{\ell}\|^{p}]) \bigr)^{2/p}.
 \end{split}
\end{align}
 Using that for all $a,b\ge0$ and $h(z)=z^{p/2}$ ($z\in[0,\infty)$) one has $h(a+b)\ge h(a)+h'(a) b$
 we get that
 \bas{
 \bigl(\E[\|\Upsilon_{\ell}\|^{p}&]^{2/p}+(p-1) \E[\|M_{\ell}\|^{p}]^{2/p}\bigr)^{p/2}\\
 &\ge \E[\|\Upsilon_{\ell}\|^{p}]+\sfrac p2 \bigl(\E[\|\Upsilon_{\ell}\|^{p}]^{2/p}\bigr)^{\frac p2-1}(p-1) \E[\|M_{\ell}\|^{p}]^{2/p}.
 }
 Thus we obtain with~(\ref{eq:3256}) that
 \begin{align}\begin{split}\label{eq:23567}
 e_{\ell}&\le \gamma_{n_{\ell}+1}^{-1} \Bigl(\bigl(\E[\|\Upsilon_{\ell}\|^{p}]^{2/p}+(p-1) \,\E[\|M_{\ell}\|^{p}]^{2/p}\bigr)^{p/2}+\sfrac12 p(p-1)\E[\|M_{\ell}\|^{p}]
 \Bigr)^{2/p}\\
 &\le  \gamma_{n_{\ell}+1}^{-1} \bigl(\E[\|\Upsilon_{\ell}\|^{p}]^{2/p}+(p-1)\, \E[\|M_{\ell}\|^{p}]^{2/p}+(\sfrac 12p(p-1))^{2/p}\,\E[\|M_{\ell}\|^{p}]^{2/p}\bigr)\\
 &\le  \gamma_{n_{\ell}+1}^{-1} \bigl(\E[\|\Upsilon_{\ell}\|^{p}]^{2/p}+\mathfrak p\, \E[\|M_{\ell}\|^{p}]^{2/p}\bigr),
 \end{split}\end{align}
 where $\mathfrak p:=p-1+(\sfrac 12p(p-1))^{2/p}$. 
%

 Next, we will provide an estimate for $\E[\|M_{\ell}\|^{p}]$. 
 First note that by Proposition~\ref{prop:234} (IV.a)
\bas{\E\bigl[ \1_{\{\mathfrak N>n_{\ell-1}\}} \bigl\|\dtilde \Theta_{n_{\ell-1},n_{\ell}}-f(\theta_{n_{\ell-1}})(t_{n_{\ell}}-t_{n_{\ell-1}})\bigr\| ^{p}\bigr]^{1/p}\le \kappa_{6} \tilde C \sqrt{\gamma_{n_{\ell-1}}(t_{n_{\ell}}-t_{n_{\ell-1}})}
}
Moreover, using estimates (I) to (III)  and the fact  that $(n_{\ell})$ is a $\rho$-partition we get that
\begin{align}\begin{split}\label{eq:135}
\E[\1_{\{\mathfrak N> n_{\ell-1}\}}&\|\theta_{n_{\ell}\wedge \mathfrak N}-\theta_{n_{\ell-1}}-\dtilde \Theta_{n_{\ell}-1:(n_{\ell}\wedge\mathfrak N)}\|^{p}]^{1/p}\\
&\le \bigl(\kappa_{1} \beta^{\frac12(n_{\ell-1}+1)-n_{0}}\|\bx\|_{\ell^{d}_\rho}
+\kappa_{2} C  \bigr) \gamma_{n_{\ell-1}+1} \beta^{n_{\ell-1}+1}+ \kappa_{3}\tilde L  (t_{n_{\ell}}-t_{n_{\ell-1}})^{2}\\
&\qquad + \gamma_{n_{\ell-1}+1} \bigl(\kappa_{4} C+\kappa_{5}\beta^{(n_{\ell-1}-n_{0})/2}\|\bx\|_{\ell_{\varrho}})\\
&\le \bigl(\underbrace{  \kappa_{3}\rho^{2}\tilde L  +(\kappa_{2}+\kappa_{4}) C+(\kappa_{1}+\kappa_{5})\beta^{(n_{\ell-1}-n_{0})/2}\|\bx\|_{\ell^{d}_{\varrho}}}_{=\aleph_{\ell}}\bigr)  \gamma_{n_{\ell-1}+1} .
\end{split}
\end{align}
By Jensen's inequality, one has that
\bas{\label{eq8346-2}
 \E[\|A_{\ell}\|^{p}]^{1/p}\le \aleph_{\ell}\gamma_{n_{\ell-1}+1}
}
so that we arrive with Lemma~\ref{le:3571} at
\bas{\label{eq72841}
\gamma_{n_{\ell}+1}^{-1}\,\E[\|M_{\ell}\|^{p}]^{2/p} &\le \bigl( \E\bigl[ \1_{\{\mathfrak N>n_{\ell-1}\}} \bigl\|\dtilde \Theta_{n_{\ell-1},n_{\ell}}-f(\theta_{n_{\ell-1}})(t_{n_{\ell}}-t_{n_{\ell-1}})\bigr\| ^{p}\bigr]^{1/p}+ 2\,\E[\|A_{\ell}\|^{p}]^{1/p}\bigr)^{2}\\
&\le \gamma_{n_{\ell}+1}^{-1}\bigl( \kappa_{6} \tilde C \sqrt{\gamma_{n_{\ell-1}+1}(t_{n_{\ell}}-t_{n_{\ell-1}})}
+2 \aleph_{\ell}\gamma_{n_{\ell-1}+1}\bigr)^{2}\\
&\le 2K\kappa_{6}^{2} \tilde C^{2} (t_{n_{\ell}}-t_{n_{\ell-1}})
+4 K\aleph_{\ell}^{2}\gamma_{n_{\ell-1}+1}\\
&\le \Bigl(2K\kappa_{6}^{2} \tilde C^{2} 
+\frac{4 K}{\rho-\sqrt{\delta_{1}}} \aleph_{\ell}^{2}\sqrt{\gamma_{n_{\ell-1}+1}}\Bigr)(t_{n_{\ell}}-t_{n_{\ell-1}}).}
Together with~(\ref{eq:23567}) 
we obtain that
\begin{align}\label{eq:23567-2}\begin{split}
 e_{\ell} &\le  \gamma_{n_{\ell}+1}^{-1} \bigl(\E[\|\Upsilon_{\ell}\|^{p}]^{2/p}+\mathfrak p\, \E[\|M_{\ell}\|^{p}]^{2/p}\bigr)\\
 &\le \gamma_{n_{\ell}+1}^{-1} \E[\|\Upsilon_{\ell}\|^{p}]^{2/p}+\mathfrak p\Bigl(2K\kappa_{6}^{2} \tilde C^{2} 
+\frac{4 K}{\rho-\sqrt{\delta_{1}}} \aleph_{\ell}^{2}\sqrt{\gamma_{n_{\ell-1}+1}}\Bigr)(t_{n_{\ell}}-t_{n_{\ell-1}}).
\end{split}\end{align}

Next, we provide an estimate for $\E[\|\Upsilon_{\ell}\|^{p}]$. 
Using that $\Psi$ is $L$-Lipschitz and $C'$-bounded on $V$ w.r.t.\ $\|\cdot\|$ we conclude that
\bas{\label{eq78234}
\|\tilde A_{\ell}\|&\le \int _{t_{n_{\ell-1}}}^{t_{n_{\ell}}} \|f(\Psi_{s})-f(\Psi_{t_{n_{\ell-1}}})\| \,\dd s \le L  \int _{t_{n_{\ell-1}}}^{t_{n_{\ell}}}\int _{t_{n_{\ell-1}}}^{s}\|f(\Psi_{u})\|\,\dd u\,\dd s
\\
&\le\sfrac 12 LC' (t_{n_{\ell}}-t_{n_{\ell-1}})^{2}.
}
Hence, we get  with~(\ref{eq8346-2}) and Lemma~\ref{le:3571} that
\begin{align}\begin{split}\label{eq893456}
E[\|A_{\ell}-\tilde A_{\ell}\|^{p}]^{1/p}&\le \aleph_{\ell}\gamma_{n_{\ell-1}+1}+\sfrac 12 LC' (t_{n_{\ell}}-t_{n_{\ell-1}})^{2}\\
&\le \Bigl(\underbrace {\frac1{\rho-\sqrt\delta} \aleph_{\ell}+\sfrac12 LC'\rho}_{=a_{\ell}}\Bigr) \sqrt{\gamma_{n_{\ell-1}+1}}(t_{n_{\ell}}-t_{n_{\ell-1}}) .
\end{split}\end{align}
As consequence of the monotonicity assumption (4) and the Lipschitz continuity of $f$, one has on  $\{\mathfrak N>n_{\ell-1}\}$ that
\bas{\|\theta_{n_{\ell-1}}&-\Psi_{t_{n_{\ell-1}}}
+(f(\theta_{n_{\ell-1}})-f(\Psi_{t_{n_{\ell-1}}}))(t_{n_{\ell}}-t_{n_{\ell-1}})\|^{2}\\
&=\|\theta_{n_{\ell-1}}-\Psi_{t_{n_{\ell-1}}}\|^{2}+ 2(t_{n_{\ell}}-t_{n_{\ell-1}}) \llangle \theta_{n_{\ell-1}}-\Psi_{t_{n_{\ell-1}}},f(\theta_{n_{\ell-1}})-f(\Psi_{t_{n_{\ell-1}}})\rrangle\\
&\qquad\qquad\qquad\qquad\qquad+(t_{n_{\ell}}-t_{n_{\ell-1}})^{2} \|f(\theta_{n_{\ell-1}})-f(\Psi_{t_{n_{\ell-1}}})\|^{2}\\
&\le (1-2c_{1}(t_{n_{\ell}}-t_{n_{\ell-1}})+L^{2}(t_{n_{\ell}}-t_{n_{\ell-1}})^{2})\,\|\theta_{n_{\ell-1}}-\Psi_{t_{n_{\ell-1}}}\|^{2}.
}
As consequence of assumption~(c) we have that  $2c_{1}-L^{2}(t_{n_{\ell}}-t_{n_{\ell-1}})\ge 0$  and we get that
\bas{
\|\Upsilon_{\ell}\|^{2}\le \1_{\{\mathfrak N>n_{\ell-1}\}}\bigl(&(1-2c_{1}(t_{n_{\ell}}-t_{n_{\ell-1}})+L^{2}(t_{n_{\ell}}-t_{n_{\ell-1}})^{2})\,\|\theta_{n_{\ell-1}}-\Psi_{t_{n_{\ell-1}}}\|^{2}\\
&+2\|\theta_{n_{\ell-1}}-\Psi_{t_{n_{\ell-1}}}\|\,\| A_{\ell} -\tilde A_{\ell} \|+\| A_{\ell} -\tilde A_{\ell} \|^{2}\bigr).
}
Now note that by assumption~(d),  we have that $a_{\ell}(t_{n_{\ell}}-t_{n_{\ell-1}})\le \delta_{2}$ for all $\ell\in\N$  so that we arrive  with~(\ref{eq893456}) and the Cauchy-Schwarz inequality at
\begin{align}\begin{split}\label{eq62467}
\E[\|\Upsilon_{\ell}\|^{p}]^{2/p}&\le  (1-2c_{1}(t_{n_{\ell}}-t_{n_{\ell-1}})+L^{2}(t_{n_{\ell}}-t_{n_{\ell-1}})^{2})\,\E[\1_{\{\mathfrak N>n_{\ell-1}\}} \|\theta_{n_{\ell-1}}-\Psi_{t_{n_{\ell-1}}}\|^{p}]^{2/p}\\
&\qquad +2 \,\E[\1_{\{\mathfrak N>n_{\ell-1}\}} \|\theta_{n_{\ell-1}}-\Psi_{t_{n_{\ell-1}}}\|^{p}]^{1/p} \,\E[\| A_{\ell} -\tilde A_{\ell} \|^{p}]^{1/p}+\E[\| A_{\ell} -\tilde A_{\ell} \|^{p}]^{2/p} \\
&\le (1-2c_{1}(t_{n_{\ell}}-t_{n_{\ell-1}})+L^{2}(t_{n_{\ell}}-t_{n_{\ell-1}})^{2}) \gamma_{n_{\ell-1}+1}e_{\ell-1}+ 2 \gamma_{n_{\ell-1}+1} \sqrt{e_{\ell-1}} a_{\ell} (t_{n_{\ell}}-t_{n_{\ell-1}})\\
&\qquad+a_{\ell}^{2} \gamma_{n_{\ell-1}+1}(t_{n_{\ell}}-t_{n_{\ell-1}})^{2}\\
&\le (1-2c_{1}(t_{n_{\ell}}-t_{n_{\ell-1}})+L^{2}(t_{n_{\ell}}-t_{n_{\ell-1}})^{2}) \gamma_{n_{\ell-1}+1}e_{\ell-1}\\
&\qquad+ \gamma_{n_{\ell-1}+1}(2 \sqrt{e_{\ell-1}}+\delta_{2}) a_{\ell} (t_{n_{\ell}}-t_{n_{\ell-1}}).
\end{split}\end{align}
We apply Lemma~\ref{le:3571} and use assumption (c) to  deduce that
\bas{
\gamma_{n_{\ell}+1}^{-1}\,\E[\|\Upsilon_{\ell}\|^{p}]^{2/p}&\le \underbrace{(1+\zeta K(t_{n_{\ell}}-t_{n_{\ell-1}})) (1-2c_{1}(t_{n_{\ell}}-t_{n_{\ell-1}})+L^{2}(t_{n_{\ell}}-t_{n_{\ell-1}})^{2})}_{\le 1-2c'(t_{n_{\ell}}-t_{n_{\ell-1}}) } e_{\ell-1}\\
&\qquad + K(2 \sqrt{e_{\ell-1}}+\delta_{2}) a_{\ell} (t_{n_{\ell}}-t_{n_{\ell-1}}).
}
We combine this estimate  with~(\ref{eq:23567-2}) and get that
\bas{
e_{\ell}\le (1-2c'(t_{n_{\ell}}-t_{n_{\ell-1}}) )\, e_{\ell-1} &+K(2 \sqrt{e_{\ell-1}} +\delta_{2}) a_{\ell}(t_{n_{\ell}}-t_{n_{\ell-1}})\\
&+\mathfrak p\Bigl(2K\kappa_{6}^{2} \tilde C^{2} 
+\frac{4 K}{\rho-\sqrt{\delta_{1}}} \aleph_{\ell}^{2}\sqrt{\gamma_{n_{\ell-1}+1}}\Bigr)(t_{n_{\ell}}-t_{n_{\ell-1}}).
}
\end{proof}

\begin{prop}[\ODE\ based uniform error analysis for the Adam optimizer] 
\label{prop35-2}
Assume 
\cref{item:1_non_uniform}, \cref{item:2_non_uniform}, \cref{item:3_non_uniform}, 
\cref{item:4_non_uniform}, \cref{item:5_non_uniform}, 
\cref{item:6a_non_uniform}, \cref{item:6b_non_uniform}
of \cref{prop35-12}, 
assume that $p>2$ and let again
\bas{
\mathfrak N=\inf\{n\ge n_{0}: \theta_{n}\not\in V\}\wedge \inf\{n_{\ell}: \ell\in\N_{0}, \|\theta_{n_{\ell}}-\Psi_{t_{n_{\ell}}}\|>\mathfrak R_{t_{n_{\ell}}}\}.
}
One has
\bas{
\label{eq:uniform_error_analysis}
\E\bigl[\sup _{n: n_{\ell-1}\le n\le n_{\ell}\wedge\mathfrak N}& \|\theta_{n}-\theta_{n_{\ell-1}}-(\Psi_{t_{n}}-\Psi_{t_{n_{\ell-1}}})\|^{p}\bigr]^{1/p}\\
&\le L ( t_{n_{\ell}}-t_{n_{\ell-1}}) \E\bigl[\1_{\{\mathfrak N\ge n_{\ell-1}\}} \|\theta_{n_{\ell-1}}-\Psi_{t_{n_{\ell-1}}}\|^{p} \bigr]^{1/p}\\
&\qquad 
+(1-2^{-(\frac12-\frac 1p)})^{-1} K^{5/4} \kappa_{6} \tilde C \gamma_{n_{\ell}+1}^{3/4} +K(\aleph_{\ell} + \sfrac12 LC' \rho^{2}) \gamma_{n_{\ell}+1} ,
}
where $K=(1-\zeta\rho\sqrt{\gamma_{n_{0}+1}})^{-1}$ and
\bas{
\aleph_{\ell}=\kappa_{3}\rho^{2}\tilde L  +(\kappa_{2}+\kappa_{4}) C+(\kappa_{1}+\kappa_{5})\beta^{(n_{\ell-1}-n_{0})/2}\|\bx\|_{\ell^{d}_{\varrho}}.
}
\end{prop}

\begin{proof}
We use the approximations $(\Theta_n),(\tilde\Theta_n)$ and $(\dtilde \Theta_n)$ as introduced in Section~\ref{sec:approx_intro} and write for $\ell\in\N$ and $n\in\{n_{\ell-1},\dots,n_{\ell}\}$
\begin{align}\label{eq9823-1}
\Psi_{t_{n}}=\Psi_{t_{n_{\ell-1}}}+(t_{n}-t_{n_{\ell-1}})f(\Psi_{t_{n_{\ell-1}}})+\tilde A_{\ell,n}
\end{align}
and 
\begin{align}\label{eq9823-2}
\theta_{n}=\theta_{n_{\ell-1}}+ \underbrace{\theta_{n}-\theta_{n_{\ell-1}}-\dtilde\Theta_{n_{\ell-1}:n}}_{=:A_{\ell,n}} +\underbrace{\dtilde\Theta_{n_{\ell-1}:n_{\ell}}- (t_{n}-t_{n_{\ell-1}}) f(\theta_{n_{\ell-1}})}_{=:M_{\ell,n}}+(t_{n}-t_{n_{\ell-1}})f(\theta_{n_{\ell-1}}).
\end{align}
In complete analogy to the proof of~(\ref{eq:135}), we get that
\bas{
\E\bigl[\max_{n=n_{\ell-1},\dots,n_{\ell}} \1_{\{\mathfrak N\ge n \}}\|A_{\ell,n}\|^{p}\bigr]^{1/p}&\le \E\Bigl[ \Bigl(\sum_{k=n_{\ell-1}+1}^{n_{\ell}}\|\Delta \theta_{k}-\Delta \dtilde \Theta_{k}\|\Bigr)^{p}\Bigr]^{1/p}\\
&\le \aleph_{\ell} \gamma_{n_{\ell-1}+1} 
}
with the identical $\aleph_{\ell}$. Moreover, Propositon~\ref{prop:234} (IV.b) yields that
\bas{
\E\bigl[\1_{\{\mathfrak N>n_{\ell-1}\}}\max_{n=n_{\ell-1},\dots,n_{\ell}} \|M_{\ell,n}\|^{p}\bigr]^{1/p}\le (1-2^{-(\frac12-\frac 1p)})^{-1}  \kappa_{6} \tilde C \gamma_{n_{\ell-1}+1} \sqrt {n_{\ell}-n_{\ell-1}}.
}
With Lemma~\ref{le:3571} we get that $\gamma_{n_{\ell-1}+1}\le K \gamma_{n_{\ell}+1}$ and
$n_{\ell}-n_{\ell-1}\le  (t_{n_{\ell}}-t_{n_{\ell-1}})/\gamma_{n_{\ell}+1}$ which entails that
\bas{
\E\bigl[\1_{\{\mathfrak N>n_{\ell-1}\}}\max_{n=n_{\ell-1},\dots,n_{\ell}} \|M_{\ell,n}\|^{p}\bigr]^{1/p}  \le (1-2^{-(\frac12-\frac 1p)})^{-1} K \kappa_{6} \tilde C \sqrt {\gamma_{n_{\ell}+1} (t_{n_{\ell}}-t_{n_{\ell}})}.
}
Moreover,
\bas{
\1_{\{\mathfrak N>n_{\ell-1}\}} \max_{n=n_{\ell-1},\dots,n_{\ell}}\bigl| (t_{n}-t_{n_{\ell-1}})
(f(\theta_{n_{\ell-1}})-f(\Psi_{t_{n_{\ell-1}}}))\bigr| \le ( t_{n_{\ell}}-t_{n_{\ell-1}}) L \|\theta_{n_{\ell-1}}-\Psi_{t_{n_{\ell-1}}}\| 
}
and $\1_{\{\mathfrak N>n_{\ell-1}\}} \max _{n=n_{\ell-1},\dots,n_{\ell}}\|\tilde A_{\ell,n}\|\le \frac12 LC'( t_{n_{\ell}}-t_{n_{\ell-1}})^{2}$ as in~(\ref{eq78234}). Combining these estimates with the representations~(\ref{eq9823-1}) and~(\ref{eq9823-2}) we get that
\bas{
\E\Bigl[\sup _{n: n_{\ell-1}\le n\le n_{\ell}\wedge\mathfrak N}& \|\theta_{n}-\theta_{n_{\ell-1}}-(\Psi_{t_{n}}-\Psi_{t_{n_{\ell-1}}})\|^{p}\Bigr]^{1/p}\le L ( t_{n_{\ell}}-t_{n_{\ell-1}}) \E\bigl[\1_{\{\mathfrak N\ge n_{\ell-1}\}} \|\theta_{n_{\ell-1}}-\Psi_{t_{n_{\ell-1}}}\|^{p} \bigr]^{1/p}\\
&\qquad+\aleph_{\ell} \gamma_{n_{\ell-1}+1} + (1-2^{-(\frac12-\frac 1p)})^{-1} K \kappa_{6} \tilde C \sqrt {\gamma_{n_{\ell}+1} (t_{n_{\ell}}-t_{n_{\ell-1}})}\\
&\qquad+ \sfrac12 LC'( t_{n_{\ell}}-t_{n_{\ell-1}})^{2}\\
&\le L ( t_{n_{\ell}}-t_{n_{\ell-1}}) \E\bigl[\1_{\{\mathfrak N\ge n_{\ell-1}\}} \|\theta_{n_{\ell-1}}-\Psi_{t_{n_{\ell-1}}}\|^{p} \bigr]^{1/p}\\
&\qquad 
+(1-2^{-(\frac12-\frac 1p)})^{-1} K^{5/4} \kappa_{6} \tilde C \gamma_{n_{\ell}+1}^{3/4} +K(\aleph_{\ell} + \sfrac12 LC' \rho^{2}) \gamma_{n_{\ell}+1} .
}
\end{proof}

\section{Proof of the main result of this article (Theorem~\ref{thm-2} in \cref{sec:main_results})}
\label{sec:proof_main_result}

In the proof of Theorem~\ref{thm-2} we proceed as follows. First we verify that Prop.~\ref{prop35-12} is applicable for sufficiently large $\mathfrak n$ and appropriate constants (steps 1.-3.). In the next step, the recursive estimate is used to deduce an error estimate at the times of a $\rho$-partition, see~(\ref{eq2641}) below. In step 5.\ this estimate is extended to all time instances in order to prove statement one of the theorem. Finally, the estimate of step 4 is used to deduce the  second statement of the theorem.

\begin{proof}[Proof of Theorem~\ref{thm-2}] {\red We provide a proof under the additional assumption that for every $x\in\R^d$ one has $\|x\|\le |x|$. The general statement then can be easily obtained by applying the result with norm $\|\cdot\|^*=\iota \, \|\cdot\|$ with $\iota\in(0,\infty)$ sufficiently small in place of $\|\cdot\|$.}

1.) \emph{Choice of parameters for the application of Prop.~\ref{prop35-12}.}
{\red We denote by $\kappa_{0}\in(0,\infty)$ the condition number of the canonical embedding of $(\R^d,|\cdot|)$ into $(\R^d,\|\cdot\|)$. 
}

We pick $\zeta\in (2c_{2},\infty)$, $\rho,c',\delta,\delta_{1},\delta_{2}\in(0,\infty)$ 
 and  $\mathfrak n\in\N_{0}$  so that the following inequalities hold:
\begin{enumerate}[label=(\roman*)]
\item  $\rho\sqrt{\gamma_{\mathfrak n+1}}<\zeta^{-1} \wedge 1$, $\gamma_{\mathfrak n+1}\le \delta_{1}<\rho^{2}$,
\item $\frac{\gamma_{n}-\gamma_{n+1}}{\gamma_{n}^{2}}\le \zeta$ for all $n\ge \mathfrak n+1$,
\item $2c_{1}-({\red \kappa_0}\cK \|\varrho\|_{\ell_{1}})^{2}\rho \sqrt{\gamma_{\mathfrak n+1}}-\zeta K\ge 2c' > 2 c$, for $K:=(1-\zeta \rho\sqrt{\gamma_{\mathfrak n+1}})^{-1}$,
\item $\displaystyle{\sfrac\rho{\rho-\sqrt{\delta_{1}}}\bigl((\kappa_{3}\rho^{2}\tilde L+\sfrac12LC'\rho^{2}+(\kappa_{2}+\kappa_{4})C)\sqrt{\gamma_{\mathfrak n+1}}+(\kappa_1+\kappa_{5})2\cK \bigr) \le \delta_{2}}$,
\item $\sfrac 12 \gamma_{\mathfrak n+1}^{-1} \log \beta^{-1}-2c\ge \delta \gamma_{\mathfrak n+1}^{-1}$,
\item $e^{c\rho \sqrt{\gamma_{\mathfrak n+1}}}\le 1+\sqrt{\gamma_{\mathfrak n+1}}$ and $\sqrt K\le 1+\eps$.
\end{enumerate}
Indeed, this is easily established: pick $\delta_{1},\rho\in(0,\infty)$ with $\delta_{1}<\rho^{2}$ and $c\rho<1$. Since $2c_{1}-2c_{2}>2c$ we can pick $\zeta>2c_{2}$ and $c'>c$ with
\bas
{
2c_{1}-\zeta > 2c'>2c.
}
Noting that $({\red \kappa_0}\cK \|\varrho\|_{\ell_{1}})^2 \rho \sqrt{\gamma_{\mathfrak n+1}}+\zeta(K-1)$ tends to zero as $\mathfrak n$ tends to infinity, we conclude that (c) holds for all but finitely many  $\mathfrak n\in\N_{0}$. The same is true for the inequalities in (a) and (b) as consequence of~(\ref{eq87236455}) and the fact that $\zeta>2c_{2}$. We pick $\delta_{2}=\sfrac\rho{\rho-\sqrt{\delta_{1}}}(\kappa_{1}+\kappa_{5})2\cK+1$ and note that since $\gamma_{\mathfrak n+1}$ tends to zero (d) holds for all but finitely many $\mathfrak n$. Pick $\delta\in(0,\frac12 \log\beta^{-1})$ and observe that (e) is satisfied for all but finitely many $\mathfrak n$. Analogously, the estimates in (f) hold  for all but finitely many $\mathfrak n$ since $(\gamma_{n})$ is a zero sequence and $K$ converges to one when letting  $\mathfrak n$ go to infinity.
Thus we have shown that (a), (b), (c), (d), (e) and (f) hold for an appropriately fixed $\mathfrak n\in\N_{0}$.

\emph{2. Regularity of the innovation.}
We verify that an innovation $(X,U)$ as in {\it 1.} of the theorem is $p$-regular with  parameter $(C,\tilde C,\tilde L):=(\cK,2\cK,\cK)$ and, on $V$, its Adam vector field $f$ is Lipschitz continuous with parameter $L:={\red \kappa_0}\cK \|\varrho\|_{\ell_{1}}$ and uniformly bounded by $C':=d\frac {1-\alpha}{\sqrt {1-\beta}}\frac1{\sqrt{1-\alpha^{2}/\beta}}$.
Indeed, by assumption, one has for $\theta,\theta'\in V$
\bas{ 
\E[|X(U,\theta)|^{2}]^{1/2}&\le \E[|X(U,\theta)|^{p}]^{1/p}\le \cK\\
 \E[|X(U,\theta)-\E[X(U,\theta)]|^{p}]^{1/p}&\le 2\E[|X(U,\theta)|^{p}]^{1/p}\le 2\cK\\
\E[|X(U,\theta)-X(U,\theta')]|^{p}]^{1/p} &\le \cK |\theta-\theta'|.
}
Moreover, using the Lipschitz continuity of $g$ (Lemma~\ref{le:23456}) we get that
\bas{
|f(\theta)-f(\theta')|&= |\E[g(X(U_{0},\theta),X(U_{-1},\theta),\dots)-g(X(U_{0},\theta'),X(U_{-1},\theta'),\dots)]|\\
&\le \E[\|(X(U_{0},\theta)-X(U_{0},\theta'),\dots) \|_{\ell^{d}_{\varrho}}]\\
&= \|\varrho\|_{\ell_{1}}  \E[|(X(U_{0},\theta)-X(U_{0},\theta')|]\le \|\varrho\|_{\ell_{1}} \cK\, |\theta-\theta'|
}
so that by recalling that $\kappa_0$ is the condition number of the canonical embedding of the two Banach spaces we obtain that
\bas{
\red \|f(\theta)-f(\theta')\|\le \kappa_0 \|\varrho\|_{\ell_{1}} \cK\, \|\theta-\theta'\|
}
The uniform bound of $g$ is proved in Lemma~\ref{le:23456}. Hence, property (1) of Proposition~\ref{prop35-12} is satisfied for $C,\tilde C,C',L,\tilde L$ as above.

\emph{3. The Adam algorithm in view of Prop.~\ref{prop35-12}.} Let $n_{0}\in\N_{0}\cap[\mathfrak n,\infty)$ and $(\theta_{n})_{n\ge n_{0}}$ an Adam algorithm as in {\it 2.} of the theorem. We pick $\bx\in\ell_{\varrho}^{d}$ satisfying~(\ref{eq34284}) with $\|\bx\|_{\ell_{\varrho}^{d}}\le 2 \|m_{n_{0}},v_{n_{0}}\|_{\ell_{\varrho}^{d}}$ so that, in particular, $\sqrt{\gamma_{n_{0}+1}}\|\bx\|_{\ell_{\varrho}^{d}}\le 2\cK$. 
Note that {\it 2.} and {\it 3.} of the theorem are identical with (3) and (4) of Prop.~\ref{prop35-12}. 
We set $K_{n_{0}}=(1-\zeta \rho\sqrt{\gamma_{n_{0}+1}})^{-1}$.
Recalling that $n_{0}\ge\mathfrak n$, $(\gamma_{n})$ is decreasing, $f$ is $L$-Lipschitz and $K_{n_{0}}\le K$ we get validity of (6.a)-(6.d) of the proposition (with $K_{n_{0}}$ in place of $K$). 

We are in the position to apply  Prop.~\ref{prop35-12}.  Let $(n_{\ell})_{\ell\in\N_{0}}$ be a $\rho$-partition  starting in $n_{0}$,
\bas{
\mathfrak N=\inf\{n\ge n_{0}: \theta_{n}\not\in V\}\wedge \inf\{n_{\ell}: \ell\in\N_{0}, \|\theta_{n_{\ell}}-\Psi_{t_{n_{\ell}}}\|>\mathfrak R_{t_{n_{\ell}}}\}
}
and, for $\ell\in\N_{0}$,
\bas{
e_{\ell}=\gamma_{n_{\ell}+1}^{-1}\,\E[\1_{\{\mathfrak N\ge n_{\ell}\}} \|\theta_{n}-\Psi_{t_{n}}\|^{p}]^{2/p}.
}
Then the proposition implies that one has, for all $\ell\in\N$, 
\bas{\label{eq24811-3}
e_{\ell}\le (1-2c'(t_{n_{\ell}}-t_{n_{\ell-1}}) )\, e_{\ell-1} +\bigl(K(2 \sqrt{e_{\ell-1}} +\delta_{2}) a_{\ell}+b_{\ell}\bigr)(t_{n_{\ell}}-t_{n_{\ell-1}}),
}
where $(a_{\ell})$, $(b_{\ell})$ and $(\aleph_{\ell})$ are as in~(\ref{eq24811-1}) and~(\ref{eq24811-2}).

\emph{4. Analysis of the recursive estimate~(\ref{eq24811-3}).}
In the following, we denote by $\eta_{1},\dots$ constants that may depend on $d$, $\alpha$, $\beta$, $p$, $(\gamma_{n})$, $\cK$, $c$, $c_{1}$, $c_{2}$ and the choices that we made for   $\mathfrak n$, $c'$, $\rho$, $\delta$,$\delta_1$ and $\delta_{2}$  in part one, but not on further specific details of the underlying algorithm or the ODE. 
By definition of $(a_{\ell})$ and  $(b_{\ell})$ and the bound $\|\bx\|_{\ell_{\varrho}^{d}}\le 2\cK \gamma_{n_{0}+1}^{-1/2}$ there exists  a finite constant $\eta_{1}$ such that, for every $\ell\in\N$,
\bas{
|a_{\ell}|\vee |b_{\ell}|\le \eta_{1}(1+ \beta^{(n_{\ell-1}-n_{0})/2}\|\bx\|_{\ell_{\varrho}^{d}}).
}
Together with~(\ref{eq24811-3}) we thus get that for an appropriate constant $\eta_{2}$ one has
\bas{\label{eq24811-4}
e_{\ell}\le (1-2c'(t_{n_{\ell}}-t_{n_{\ell-1}}) )\, e_{\ell-1} &+\eta_{2} (\sqrt {e_{\ell-1}}+1)
(1+\beta^{(n_{\ell-1}-n_{0})/2}\|\bx\|_{\ell^{d}_{\varrho}})(t_{n_{\ell}}-t_{n_{\ell-1}}).
}
Clearly,  $\eta_{3}:= \sup_{z\in[0,\infty)} (\eta_{2}(\sqrt z +1)-(2c'-2c)z)$ is finite and using that $\eta_{2} (\sqrt {e_{\ell-1}}+1)\le (2c'-2c)e_{\ell-1}+\eta_{3}$ we get that
\bas{
e_{\ell}\le (1-2c(t_{n_{\ell}}-t_{n_{\ell-1}}) )\, e_{\ell-1} +\eta_{3}(t_{n_{\ell}}-t_{n_{\ell-1}})+ \eta_{2}(\sqrt {e_{\ell-1}}+1)
\beta^{(n_{\ell-1}-n_{0})/2}\|\bx\|_{\ell^{d}_{\varrho}} (t_{n_{\ell}}-t_{n_{\ell-1}})
}
and
\bas{
e_{\ell}-\sfrac{\eta_3}{2c}\le (1-2c(t_{n_{\ell}}-t_{n_{\ell-1}}) )\, (e_{\ell-1}-\sfrac{\eta_3}{2c})+ \eta_{2}(\sqrt {e_{\ell-1}}+1)
\beta^{(n_{\ell-1}-n_{0})/2}\|\bx\|_{\ell^{d}_{\varrho}} (t_{n_{\ell}}-t_{n_{\ell-1}}).
}
This gives
\begin{align}\begin{split}\label{eq87235}
e_{\ell}-\sfrac{\eta_3}{2c}&\le (e_{0}-\sfrac {\eta_{3}}{2c})_{+}e^{-2c(t_{n_{\ell}}-t_{n_{0}})}\\
&\qquad + \sum_{r=1}^{\ell}
e^{-2c(t_{n_\ell}-t_{n_r})} (\sqrt {e_{r-1}}+1)
\beta^{(n_{r-1}-n_{0})/2}\|\bx\|_{\ell^{d}_{\varrho}} (t_{n_{r}}-t_{n_{r-1}})
\end{split}\end{align}
It remains to provide an estimate for the latter sum. First note that
\bas{
n_{r-1}-n_{0}\ge \frac{t_{n_{r-1}}-t_{n_{0}}}{\gamma_{n_{0}+1}}
}
and as consequence of (e) above we get  that 
\bas{
\frac1{2\gamma_{n_{0}+1}}\log \beta^{-1}-2c\ge \delta \gamma_{n_{0}+1}^{-1}>0.
}
Consequently, we have
\bas{\label{eq983456}
\sum_{r=1}^{\ell}&
e^{-2c(t_{n_\ell}-t_{n_r})}
\beta^{(n_{r-1}-n_{0})/2}
(t_{n_{r}}-t_{n_{r-1}})\\
&\le e^{2c\rho \sqrt{\gamma_{n_{0}+1}}}\sum_{r=1}^{\ell}
\underbrace{e^{-2c(t_{n_\ell}-t_{n_{r-1}})-\frac 12\log\beta^{-1}(n_{r-1}-n_{0})}}_{\le e^{-2c(t_{n_\ell}-t_{n_{0}})}e^{-\delta\gamma_{n_{0}-1}^{-1}(t_{n_{r-1}}-t_{n_{0}})}}
(t_{n_{r}}-t_{n_{r-1}})\\
&\le e^{-2c(t_{n_\ell}-t_{n_{0}}-\rho\sqrt{\gamma_{n_{0}+1}})} \sum_{r=1}^{\ell}e^{-\delta\gamma_{n_{0}+1}^{-1}(t_{n_{r-1}}-t_{n_{0}})}(t_{n_{r}}-t_{n_{r-1}})\\
&\le e^{-2c(t_{n_\ell}-t_{n_{0}}-\rho\sqrt{\gamma_{n_{0}+1}})} \Bigl(
\rho\sqrt{\gamma_{n_{0}+1}}+\int_0^{\infty} e^{-\delta\gamma_{n_{0}+1}^{-1}s}\,\dd s\Bigr)\\
&\le \eta_{4} \sqrt{\gamma_{n_{0}+1}}e^{-2c(t_{n_\ell}-t_{n_{0}})} 
}
for the constant 
$\eta_{4}=e^{2c\rho\sqrt{\gamma_{\mathfrak n+1}}}  \bigl(
\rho+\delta^{-1} \sqrt{\gamma_{\mathfrak n+1}}\bigr).
$ 
Together with~(\ref{eq87235}) we get  an a priori bound for $\mathcal M_{L}:=\sup_{\ell=0,\dots,L}e_{\ell}$ ($L\in\N_{0}$): one has for every $\ell=0,\dots,L$ that 
\bas{
e_{\ell}-\sfrac{\eta_3}{2c}\le (e_{0}-\sfrac{\eta_{3}}{2c})_{+}+(\sqrt {\mathcal M_{L}} +1) \|\bx\|_{\rho_{\ell}^d} \underbrace{\sum_{r=1}^{\ell}
e^{-2c(t_{n_\ell}-t_{n_r})}
\beta^{(n_{r-1}-n_{0})/2}
(t_{n_{r}}-t_{n_{r-1}})}_{\le \eta_{4}\sqrt{\gamma_{n_{0}+1}}}
}
so that
\bas{
\mathcal M_{L}-\underbrace{\eta_4 \sqrt {\gamma_{n_{0}+1}} \|\bx\|_{\ell_{\varrho}^d} }_{=:a}(\sqrt {\mathcal M_{L}} +1)-\bigl(\underbrace{e_{0}\vee\sfrac{\eta_{3}}{2c}}_{=:b}\bigr)\le0.
}
Hence, $\sqrt {\mathcal M_{L}}$ is smaller or equal to the right zero of the parabola $z\mapsto z^{2}-a(z+1)-b$. This gives that
\bas{
\mathcal M_{L}\le \Bigl(\frac a2+\sqrt{\frac {a^{2}}4+a+b}\Bigr)^{2}\le \Bigl(\frac a2+\sqrt{(\sfrac12 a+1)^{2}+b}\Bigr)^{2}\le (a+1+\sqrt b)^{2}.
}
By assuming that  $\mathfrak n$ is sufficiently large (otherwise we may enlarge it appropriately)  we can  guarantee that $\eta_3/(2c)\le \cK^{2}\gamma_{\mathfrak n+1}^{-1}$. Additionally, by assumption~(\ref{eq87345}),   $e_{0}\le \cK^{2} \gamma_{n_{0}+1}^{-1}$ so that $\sqrt b\le \cK\gamma_{n_{0}+1}^{-1/2}$. Together with 
 $a\le \eta_{4}\cK $ we get that
$
\sqrt {\mathcal M_{L}}+1\le \mathcal K\gamma_{n_{0}+1}^{-1/2}+ \eta_{4}\cK+2. 
$ Assuming that $\mathfrak n$ is sufficiently large that $\eta_4 \cK+2\le \cK \gamma_{\mathfrak n+1}^{-1/2}$ we obtain that
\bas{
\sqrt {\mathcal M_{L}}+1\le 2\mathcal K\gamma_{n_{0}+1}^{-1/2}.
}
Note that the latter estimate does not depend on the choice of $L$ so that it remains valid when replacing $\mathcal M_{L}$ by $\mathcal M=\sup_{\ell\in\N_{0}}
e_{\ell}$.

Insertion of this a priori bound into~(\ref{eq87235}) and using again~(\ref{eq983456}) gives that
\bas{
e_{\ell}-\sfrac{\eta_3}{2c}&\le (e_{0}-\sfrac{\eta_{3}}{2c})_{+}e^{-2c(t_{n_{\ell}}-t_{n_{0}})}\\
&\qquad +2\cK \gamma_{n_{0}+1}^{-1/2} \|\bx\|_{\ell^{d}_{\varrho}}  \sum_{r=1}^{\ell}
e^{-2c(t_{n_\ell}-t_{n_r})} 
\beta^{(n_{r-1}-n_{0})/2}(t_{n_{r}}-t_{n_{r-1}})\\
&\le  (e_{0}-\sfrac{\eta_{3}}{2c})_{+}e^{-2c(t_{n_{\ell}}-t_{n_{0}})}+ 2\cK\eta_{4} \|\bx\|_{\ell^{d}_{\varrho}}  e^{-2c(t_{n_{\ell}}-t_{n_{0}})}.
}
Consequently,
\begin{align}\begin{split}\label{eq2641}
\E[&\1_{\{\mathfrak N\ge n_{\ell}\}} \|\theta_{n_{\ell}}-\Psi_{t_{n_{\ell}}}\|^{p}]^{1/p}=\sqrt{\gamma_{n_{\ell}+1}e_{\ell}}\\
&\le \Bigl(\sqrt{\sfrac{\eta_{3}}{2c}} + e^{-c(t_{n_{\ell}}-t_{n_{0}})} \Bigl(\frac1{\sqrt{\gamma_{n_{0}+1}}}\E[\1_{\{\mathfrak N> n_0\}}\|\theta_{n_{0}}-\Psi_{t_{n_{0}}}\|^{p}]^{1/p}+\sqrt{2\cK\eta_{4}} \sqrt{\|\mathbf x\|_{\ell_{\varrho}^{d}}}  
\Bigr)\Bigr)\sqrt{\gamma_{n_{\ell}+1}}\\
&=:\bar e_{\ell}
\end{split}\end{align}

\emph{5. Proof of the first statement.}
To extend this estimate on the set of all time instances first observe that as consequence of property (f) one has
\bas{
e^{-c(t_{n_{\ell-1}}-t_{n_{0}})}  \le e^{-c(t_{n_{\ell}}-t_{n_{0}})}  e^{c\rho\sqrt{\gamma_{n_{\ell-1}+1}}}\le e^{-c(t_{n_{\ell}}-t_{n_{0}})}  (1+ \sqrt{\gamma_{n_{0}+1}})
}
and
\bas{
\sqrt{\gamma_{n_{\ell-1}+1}}\le \sqrt K \sqrt{\gamma_{n_{\ell}+1}}\le (1+\eps)\sqrt{\gamma_{n_{\ell}+1}}.
}
This entails that 
\bas{
\label{eq9872356}
\bar e_{\ell-1} 
&
  \le (1+\eps)\Bigl(\sqrt{\sfrac{\eta_{3}}{2c}} + e^{-c(t_{n_{\ell}}-t_{n_{0}})}(1+\sqrt {\gamma_{n_{0}+1}}) \Bigl(\frac1{\sqrt{\gamma_{n_{0}+1}}}\E[\1_{\{\mathfrak N> n_0\}}\|\theta_{n_{0}}-\Psi_{t_{n_{0}}}\|^{p}]^{1/p}
\\ &
\quad 
  + \sqrt{2\cK\eta_{4}} \sqrt{\|\mathbf x\|_{\ell_{\varrho}^{d}}}  
\Bigr)\Bigr)\sqrt{\gamma_{n_{\ell}+1}}
\\
&
  \le (1+\eps)\bar e_{\ell}+ \eta_{5} \sqrt{\gamma_{n_{\ell}+1}},
}
where $\eta_{5}$ is an appropriate constant.
By Proposition~\ref{prop35-2}, there is a constant $\eta_{6}$ such that
\bas{
\E\Bigl[\sup _{n: n_{\ell-1}\le n\le n_{\ell}\wedge\mathfrak N} \|\theta_{n}-\theta_{n_{\ell-1}}-(\Psi_{t_{n}}-\Psi_{t_{n_{\ell-1}}})\|^{p}\Bigr]^{1/p}&\le \eta_{6} \sqrt{\gamma_{n_{\ell}+1}} (\bar e_{\ell-1}+1)\\
&\le \eta_{6} \sqrt{\gamma_{n_{\ell}+1}} ((1+\eps)\bar e_{\ell}+\eta_{5} \sqrt{\gamma_{n_{\ell}+1}}+1)
}
Assuming that $\mathfrak n$ is sufficiently large to guarantee that $\eta_{6}\sqrt{\gamma_{\mathfrak n+1}}\le \eps$  and $\eta_{5}\sqrt{\gamma_{\mathfrak n+1}}\le 1$ we get that
\bas{
\E\Bigl[\sup _{n: n_{\ell-1}\le n\le n_{\ell}\wedge\mathfrak N} \|\theta_{n}-\theta_{n_{\ell-1}}-(\Psi_{t_{n}}-\Psi_{t_{n_{\ell-1}}})\|^{p}\Bigr]^{1/p}&\le \eps (1+\eps)\bar e_{\ell}+2\eta_{6} \sqrt{\gamma_{n_{\ell}+1}}.
}
Combining this estimate
with the triangle inequality and~(\ref{eq9872356}) we conclude that
\bas{
\E\Bigl[\sup_{n_{\ell-1}\le n\le n_{\ell}\wedge \mathfrak N}& \|\theta_{n}-\Psi_{t_{n}}\|^{p}\Bigr]^{1/p}\le  \E[\1_{\{\mathfrak N\ge n_{\ell}\}}  \|\theta_{n_{\ell-1}}-\Psi_{t_{n_{\ell-1}}}\|^{p}]^{1/p}\\
&\qquad\qquad\qquad\qquad + \E\bigl[\sup _{n: n_{\ell-1}\le n\le n_{\ell}\wedge\mathfrak N}  \|\theta_{n}-\theta_{n_{\ell-1}}-(\Psi_{t_{n}}-\Psi_{t_{n_{\ell-1}}})\|^{p}\bigr]^{1/p}\\
&\le (1+\eps)^{2}\bar e_{\ell}+ (\eta_{5}+2\eta_{6}) \sqrt{\gamma_{n_{\ell}+1}}\\
&\le \Bigl(\eta_{7}+e^{-c(t_{n_{\ell}}-t_{n_{0}})}\Bigl((1+\eps)^{2} \frac{|\theta_{n_{0}}-\Psi_{t_{n_{0}}}|}{\sqrt{\gamma_{n_{0}+1}}} +\eta_{7}  \sqrt{(m_{0},v_{0})_{\ell_{\varrho}^{d}}}\Bigr)\Bigr)\sqrt{\gamma_{n_{\ell}+1}},
}
where $\eta_{7}$ is an appropriate constant. This proves the first statement  by noticing that, in particular, for $n\ge n_{0}$ one can choose $\ell\in\N$ with $n\in\{n_{\ell-1},\dots,n_{\ell}\}$ and get that
\bas{
\E\bigl[\1_{\{ \mathfrak N\ge n\}}& \|\theta_{n}-\Psi_{t_{n}}\|^{p}\bigr]^{1/p}\le (1+\eps)^{2}\bar e_{\ell}+ (\eta_{5}+2\eta_{6}) \sqrt{\gamma_{n_{\ell}+1}}\\
&\le \Bigl(\eta_{7}+e^{-c(t_{n_{\ell}}-t_{n_{0}})}\Bigl((1+\eps)^{2} \frac{\|\theta_{n_{0}}-\Psi_{t_{n_{0}}}\|}{\sqrt{\gamma_{n_{0}+1}}} +\eta_{7}  \sqrt{(m_{0},v_{0})_{\ell_{\varrho}^{d}}}\Bigr)\Bigr)\sqrt{\gamma_{n+1}}.
}

\emph{6. Proof of the second statement.} For ease of notation, we write in this part \bas{a:=\eta_{7}\text{ \ and \ }  b:=(1+\eps)^{2}\frac{\|\theta_{n_{0}}-\Psi_{t_{n_{0}}}\| }{\sqrt{\gamma_{n_{0}+1}}}+\eta_{7}(m_{n_{0}},v_{n_{0}})_{\ell_{\varrho}^{d}}.
}
By assumption,  $(R_{t})_{t\ge t_{n_{0}}}$ is  decreasing which implies that 
\bas{
\E\Bigl[&\sum_{\ell=1}^{\infty} R_{t_{n_{\ell-1}}}^{-p} \sup_{n_{\ell-1}\le n\le n_{\ell}\wedge \mathfrak N} \|\theta_{n}-\Psi_{t_{n}}\|^{p} \Bigr]\\
&\le  \sum_{\ell=1}^{\infty} R_{t_{n_{\ell-1}}}^{-p}\bigl(a+b e^{-c(t_{n_{\ell}}-t_{n_{0}})} )^{p}\,\gamma_{n_{\ell}+1}^{p/2}\\
&\le (\rho-\sqrt{\gamma_{\mathfrak n+1}})^{-1}\sum_{\ell=1}^{\infty} R_{t_{n_{\ell-1}}}^{-p}\bigl(a+b e^{-c(t_{n_{\ell}}-t_{n_{0}})} )^{p}\gamma_{n_{\ell}+1}^{(p-1)/2} (t_{n_{\ell}}-t_{n_{\ell-1}})\\
&\le (\rho-\sqrt{\gamma_{\mathfrak n+1}})^{-1}\int_{t_{n_{0}}}^{\infty} R_{s}^{-p}\bigl(a+b e^{-c(s-t_{n_{0}})} )^{p} \,\Gamma_{s}^{(p-1)/2}\,\dd s.
}
By (a), we have that $t_{n_{\ell}}-t_{n_{\ell-1}}\le 1$ so that
\bas {
\E\Bigl[\sup_{n=n_{0},\dots,\mathfrak N} R_{t_{n}}^{-p}  \|\theta_{n}-\Psi_{t_{n}}\|^{p} \Bigr]&\le \cK^{p}\, \E\Bigl[\sum_{\ell=1}^{\infty} R_{t_{n_{\ell-1}}}^{-p} \sup_{n_{\ell-1}\le n\le n_{\ell}\wedge \mathfrak N} \|\theta_{n}-\Psi_{t_{n}}\|^{p} \Bigr]\\
&\le  \cK^{p}\, (\rho-\sqrt{\gamma_{\mathfrak n+1}})^{-1}\int_{t_{n_{0}}}^{\infty} R_{s}^{-p}\bigl(a+b e^{-c(s-t_{n_{0}})} )^{p} \,\Gamma_{s}^{(p-1)/2}\,\dd s
}
which implies the second statement of the theorem.
\end{proof}

\section{Regularity properties and perturbation analysis for the Adam vector field}
\label{sec:perturbation_analysis}

The Adam vector field as introduced in \cref{def:Adamfield0} depends in an intricate way on the innovation $(X,U)$ and the damping parameters $(\alpha,\beta,\epsilon)$.  In this section, we want to investigate the dominant terms for particular choices of parameters.

 In the following, let $(X,U)$ be an innovation and let $(\alpha,\beta,\epsilon)$ be damping parameters with $0\le\alpha<\sqrt \beta<1$. The related Adam vector field  $f \colon \R^d\to\R^d$ is given by
\bas{
f^{(i)}(\theta)&=(1-\alpha) \,\E\Bigl[ 1/\Bigl(\sqrt{(1-\beta) \sum_{k\in-\N_0} \beta^{-k} X^{(i)}(U_k,\theta)^{2}}+\epsilon\Bigr)  \sum_{k\in-\N_0} \alpha^{-k} X^{(i)}(U_k,\theta)\Bigr].
}
Motivated by a Taylor approximation we will introduce a vector field 
$  
  \tilde f 
  = ( \tilde f^{ (1) }, \dots, \tilde f^{ (d) } )
  \colon 
  \allowbreak 
  \R^d\to\R^d
$ 
which we call the \emph{first order approximation} for the Adam field $f$. For this we need some more notation.

We denote by $(U_k)_{k\in-\N_0}$ an i.i.d.\ sequence of copies of $U$ and briefly write, for all $k\in-\N_0$ and $\theta\in\R^d$, $X_k^\theta=X(U_k,\theta)$ and $\bX^\theta=(X_{k}^\theta)_{k\in-\N_0}$. Again we write for $\bx=(x_k)_{k\in-\N_0}\in\ell_\varrho$
\bas{
V(\bx)=(1-\beta)\sum_{k\in-\N_0} \beta^{-k} x_k^2
}
and let 
\bas{\label{def:h2}
h:[0,\infty)\to(0,\epsilon^{-1}], \ x\mapsto \frac1{\sqrt x+\epsilon}.}
In this notation, we have
\bas{
f^{(i)}(\theta)
&=(1-\alpha)\sum_{k\in-\N_0}  \alpha^{-k} \,\E\bigl[ h(V(\bX^{\theta, (i)})) \,  X^{\theta,(i)}_k\bigr].
}
For $k\in-\N_{0}$ and $\mathbf x=(x_r)_{r\in-\N_0}\in \ell_\varrho$, we let 
\bas{
V^{\not=k}(\mathbf x)=(1-\beta)\sum_{r\in-\N_{0}\backslash \{k\}} \beta^{-r} x_r^{2},\text{ \ for all $i=1,\dots,d$,}
}
and we set, for all $\theta\in\R^d$ and $i=1,\dots,d$,
\bas{\label{def:linApp}
&
\tilde f^{(i)}(\theta)
=(1-\alpha)\sum_{k\in-\N_0}  \alpha^{-k} \,\E\Bigl[ \frac1{\sqrt{V^{\not=k}(\bX^{\theta,(i)})}+\epsilon} \Bigr] \,\E[X^{\theta,(i)}_0]\\
& -(1-\alpha)(1-\beta)\sum_{k\in-\N_0}   (\alpha\beta)^{-k} \,\E\Bigl[   \frac1{(\sqrt{V^{\not=k}(\bX^{\theta, (i)})}+\epsilon)^{2}}\frac1{2\sqrt{V^{\not=k}(\bX^{\theta,(i)})}} \Bigr] \E\bigl[(X^{\theta,(i)}_0)^{3}\bigr]
}
provided that the latter two expectations are well-defined and finite. Note that for every $k\in-\N_0$, the random variable $V^{\not=k}(\bX^{\theta,(i)})$ stochastically dominates  $V^{\not=0}(\bX^{\theta,(i)})$. Moreover, the distribution of $V^{\not=0}(\bX^{\theta,(i)})$ agrees with the one of $\beta V(\bX^{\theta,(i)})$. 
Consequently, $\tilde f^{(i)}(\theta)$ is well-defined and finite, if and only if 
\bas{
 X_0^{\theta,(i)}\in L^3\text{ \ and \ } V(\bX^{\theta,(i)})^{-1/2}\in L^1.
}
We call $\tilde f^{(i)}(\theta)$ as in~(\ref{def:linApp}) \emph{first order approximation} for the  Adam vector field with innovation $(X,U)$ and damping parameters $(\alpha,\beta,\epsilon)$ provided that it exists.

%
%

\begin{theorem}
\label{thm:perturbation_analysis}
Let $(X,U)$ be an innovation and $(\alpha,\beta,\epsilon)$ damping parameters with $0\le\alpha<\sqrt \beta<1$. Let $\theta\in\R^d$ and $i=1,\dots,d$ and suppose that
\bas{\label{eq:23876354}
X_0^{\theta,(i)}\in L^5\text{ \ and \ } V(\bX^{\theta,(i)})^{-1/2}\in L^3.}
Then the first order approximation $\tilde f^{(i)}(\theta)$ for the Adam field with innovation $(X,U)$ and damping factors $(\alpha,\beta,\epsilon)$ is well-defined in $\theta$ and $i$ and one has that
\bas{
|f^{(i)}(\theta)-\tilde f^{(i)}(\theta)|
&\le\frac38\frac{(1-\alpha) (1-\beta)^{2}}{\beta^{5/2}(1-\alpha\beta^{2})} \,  \E\Bigl[\frac1{(\sqrt{V(\bX^{\theta,(i)})}+\epsilon)^{2}\,V(\bX^{\theta,(i)})^{3/2}} \Bigr] \,\E^{\theta}[ |X^{\theta,(i)}_0|^5\bigr].
}
\end{theorem}

\begin{rem}\label{rem:non-convergence}
Supposing that~\cref{eq:23876354} is satisfied for given $\theta\in\R^d$ and $i\in\{1,\dots,d\}$ we can indeed use the previous theorem to show that $\tilde f^{(i)}(\theta)$ approximates $f^{(i)}(\theta)$ well, when fixing $\alpha$ and $\epsilon$ and letting $\beta\to1$. In this setting, $|f^{(i)}(\theta)-\tilde f^{(i)}(\theta)|$ is of order $(1-\beta)^2$ and the first term on the right-hand side of~(\ref{def:linApp}) converges and the second term is of order $1-\beta$. Thus $\tilde f^{(i)}(\theta)$ consists of the two leading order terms. In particular, the second order term gives insights about the deviation of critical points of the Adam field  from critical points of the vector field $\theta\mapsto \E[X(U,\theta)]$ which is the field related to classical gradient descent.
\end{rem}

\begin{proof}
Note that one has that, for all $x\in[0,\infty)$, 
\bas{
h'(x)=-\frac1{(\sqrt x+\epsilon)^{2}}\frac1{2\sqrt x}\text{ \ and  \ }h''(x)=\frac1{(\sqrt x+\epsilon)^{3}}\frac1{2x}+\frac1{(\sqrt x+\epsilon)^{2}}\frac1{4 x^{3/2}}.
}
Observe that one has
\bas{
(1-\alpha)&\sum_{k\in-\N_0} \alpha^{-k}  \,\E\bigl[\bigl(h(V^{\not=k}(\bX^{\theta,(i)}) ) +h'(V^{\not=k}(\bX^{\theta,(i)})) (1-\beta)\beta^{-k} (X_k^{\theta,(i)})^2\bigr) \,X_{k}^{\theta,(i)}\bigr]\\
&=(1-\alpha)\sum_{k\in-\N_0}  \alpha^{-k} \,\E\Bigl[ \frac1{\sqrt{V^{\not=k}(\bX^{\theta,(i)})}+\epsilon} \Bigr] \,\E[X^{\theta,(i)}_0]\\
&\quad -(1-\alpha)(1-\beta)\sum_{k\in-\N_0}   (\alpha\beta)^{-k} \,\E\Bigl[   \frac1{(\sqrt{V^{\not=k}(\bX^{\theta, (i)})}+\epsilon)^{2}}\frac1{2\sqrt{V^{\not=k}(\bX^{\theta,(i)})}} \Bigr] \E\bigl[(X^{\theta,(i)}_0)^{3}\bigr]
\\
&=\tilde f^{(i)}(\theta),
}
where we used independence of $V^{\not=k}(\bX^{\theta,(i)})$ and $X_k^{\theta,(i)}$ in the previous step.
Using that for all $k\in\-\N_0$, $V(\bX^{\theta,(i)})=V^{\not=k}(\bX^{\theta,(i)})+(1-\beta)\beta^{-k} (X_k^{(i)})^2$, we get that
\bas{
h&(V^{\not=k}(\bX^{\theta,(i)})) +h'(V^{\not=k}(\bX^{\theta,(i)}))(1-\beta) \beta^{-k} (X^{\theta,(i)}_k)^2  \le h(V(\bX^{\theta,(i)})) \\
&\le h(V^{\not=k}(\bX^{\theta,(i)})) +h'(V^{\not=k}(\bX^{\theta,(i)}))(1-\beta) \beta^{-k} (X^{\theta,(i)}_k)^2  \\
&\qquad\qquad + \frac 12 h''(V^{\not=k}(\bX^{\theta,(i)})) (1-\beta)^{2} \beta^{-2k}(X^{\theta,(i)}_k) ^4.
}
Consequently,
\bas{
|f^{(i)}(\theta)-\tilde f^{(i)}(\theta)|&\le (1-\alpha)\sum_{k\in-\N_0}   \alpha^{-k} \,\frac12 (1-\beta)^{2} \beta^{-2k} \,\E\bigl[h''(V^{\not=k}(\bX^{\theta,(i)})) |X^{\theta,(i)}_k|^5\bigr]\\
&=(1-\alpha)\sum_{k\in-\N_0}   \alpha^{-k} \,\frac12 (1-\beta)^{2} \beta^{-2k} \,\E\bigl[h''(V^{\not=k}(\bX^{\theta,(i)}))\bigr] \,\E\bigl[ |X^{\theta,(i)}_0|^5\bigr]
}
Using that $V^{\not =k}(\bX^{\theta,(i)})$ stochastically dominates $\beta \,V(\bX^{\theta,(i)})$, the monotonicity of $h''$ and the fact that
$
h''(x)\le \frac 34 \frac1{(\sqrt x+\epsilon)^{2}x^{3/2}}
$ 
we conclude that
\bas{
|f^{(i)}(\theta)-\tilde f^{(i)}(\theta)|&\le (1-\alpha)\sum_{k\in-\N_0}  \alpha^{-k} \,\frac12 (1-\beta)^{2} \beta^{-2k} \,\E\bigl[h''(\beta V(\bX^{\theta,(i)}))\bigr] \,\E\bigl[ |X^{\theta,(i)}_0|^5\bigr]\\
&\le\frac38\frac{(1-\alpha) (1-\beta)^{2}}{\beta^{5/2}(1-\alpha\beta^{2})} \,  \E\Bigl[\frac1{(\sqrt{V(\bX^{\theta,(i)})}+\epsilon)^{2}\,V(\bX^{\theta,(i)})^{3/2}} \Bigr] \,\E^{\theta}[ |X^{\theta,(i)}_0|^5\bigr].
}
\end{proof}

In the second part of this section we investigate the Adam field when working with mini-batches. The crucial result will be a first order approximation when letting the size of the mini-badges tend to infinity.
For a fixed innovation $(X,U)$ be an innovation, we let  $\IU=(U_r)_{r\in\N}$ be a sequence of independent copies of $U$ and for every $M\in\N$, we call the innovation $(\mathbb X_M,\IU)$ with
\bas{
\mathbb X_M(u)=\frac 1M\sum_{r=1}^M X(u_r,\theta)
}
the \emph{mini-batch innovation of $(X,U)$ of size $M$}.

\begin{theorem}\label{theo:732846}Let $(X,U)$ be an innovation and $(\alpha,\beta,\epsilon)$ be damping parameters with $\alpha<\sqrt\beta<1$. Let $M\in\N$ and $(\mathbb X_M,\IU)$ be the   mini-batch innovation of $(X,U)$ of size $M$. Moreover, let $\theta\in\R^d$ and $i\in\{1,\dots,d\}$ with 
\bas{
\E[X^{(i)}(U,\theta)]=0 
\qquad
\text{and}
\qquad 
\E[|h'(\IV_{M})|]<\infty . 
}
Set for $q\in(0,\infty)$, $\varphi^{(i)}_q=\E[|X^{(i)}(U,\theta)|^q]$.
Then for the corresponding Adam vector field $f_M$, one has that
\bas{
&
|f_M^{(i)}(\theta)| \le \beta^{-3/2} \frac{(1-\alpha)(1-\beta)}{1-\alpha\beta}\varphi_3\, 
\E[|h'( \IV_M^{(i)})|]\,M^{-2} 
\\
& 
  + \beta^{-5/2}\frac{(1-\alpha)(1-\beta)^2}{1-\alpha\beta^2}\E[h''( \IV_M^{(i)})] (2 C_3 \varphi_2^{(i)} \varphi^{(i)}_3 M^{-5/2}+ 5  \varphi_2 ^{(i)}\varphi_3^{(i)} M^{-3}+\varphi_5^{(i)} M^{-4}),
}
where $C_3$ is the constant in the Burkholder-Davis-Gundy inequality for the third moment.
\end{theorem}
In our proof of \cref{theo:732846} we will use the following lemma.
\begin{lemma}\label{le:3465}
Let $A,B$ and $C$ be  independent real random variables with $A>0$, almost surely, and let $h:[0,\infty)\to[0,\epsilon^{-1})$ be as in~(\ref{def:h2}). 
Provided that $\E[|h'(A)|]<\infty$ one has that
\bas{
&
\bigl|\E\bigl[h(A+(B+ C)^2) C\bigr]-\bigl(\E[h(A+B^2)]\,\E[C] +\E[h'(A)] \bigl(2\E[B]\,\E[ C^2] +\E[ C^3]\bigr)\bigr)\bigr|
\\
& \le 
  \E[ h''(A)] \bigl(2\E[|B|^3]\,\E[C^2]+5\E [B^2] \,\E[|C|^3]+\E[|C|^5]\bigr).
}
In particular, in the case where $\E[B]=\E[C]=0$, one has that
\bas{\label{eq:34572}
\bigl|\E\bigl[h&(A+(B+ C)^2) C\bigr]-\E[h'(A)] \,\E[ C^3]\bigr|\\
&\le \E[ h''(A)] \bigl(2\E[|B|^3]\,\E[C^2]+5\E [B^2] \,\E[|C|^3]+\E[|C|^5]\bigr).
}
\end{lemma}

\begin{proof}
First note that
\bas{
h(A+(B+ C)^2) C&=h(A+B^2+2 B C+ C^2) C\\
&= \bigl(h(A+B^2)+h'(A+B^2)(2 B C+ C^2)+\psi_1\bigr) C,
}
where as consequence of the Taylor formula and the monotonicity of $h''$,  $|\psi_1|\le\frac12 h''(A)(2 B C+C^2)^2$. Moreover,
\bas{
h'(A+B^2)= h'(A) +\psi_2,
}
where $|\psi_2|\le  h''(A) B^2$. Consequently,
\bas{
\bigl|h&(A+(B+ C)^2) C-  \bigl(h(A+B^2)+h'(A)(2 B C+ C^2)\bigr) C\bigr|\\
&=|\psi_1 C+\psi_2(2 BC+ C^2)C |\le h''(A) \bigl(2|B|^3C^2+5 B^2|C|^3+|C|^5\bigr)
}
and by independence of $A$, $B$ and $C$
 \bas{
\bigl|\E\bigl[h&(A+(B+ C)^2) C\bigr]-\E\bigl[  \bigl(h(A+B^2)+h'(A)(2 B C+ C^2)\bigr) C\bigr]\bigr|\\
&\le \E[ h''(A)] \bigl(2\E[|B|^3]\,\E[C^2]+5\E [B^2] \,\E[|C|^3]+\E[|C|^5]\bigr).
}
The statement follows immediately by using the independence of $A$, $B$ and $C$. 
\end{proof}

Using \cref{le:3465} we now present 
the proof of \cref{theo:732846}.

\begin{proof}
For the proof we can assume without loss of generality that $d=1$. 
We denote by $(\IU(k))_{k\in-\N_0}$ a sequence of independent random variables each one being identically distributed as~$\IU$ and we 
 set for $k\in-\N_0$ and $r\in \N$
\bas{
X_{k,r}=X(\IU_r(k),\theta) \text{ \ and \ } \IV_M=(1-\beta) \sum_{k\in-\N_0} \beta^{-k} \Bigl(\frac1M\sum_{r=1}^M X_{k,r}\Bigr)^2.
}
Then 
\bas{\label{eq:23578563-1}
f_M(\theta)&=(1-\alpha)\,\E\Bigl[ h(\IV_M) \sum _{k\in-\N_0}\alpha^{-k} \frac 1M \sum _{r=1}^M X_{k,r}\Bigr] \\
&=(1-\alpha) \sum _{k\in-\N_0}\alpha^{-k}  \sum _{r=1}^M \E\bigl[ h(\mathbb V_M) \sfrac 1M X_{k,r}\bigr].
}
We let
\bas{
A_{k,r}=(1-\beta)\sum_{\ell\in-\N_0\backslash\{k\}} \beta ^{-\ell}\Bigl( \frac1M \sum_{s=1}^M X_{\ell,s}\Bigr)^2, \ B_{k,r}= \sqrt{(1-\beta) \beta^{-k}}   \frac1M \sum_{s\in\{1,\dots,M\}\backslash\{r\}}^M X_{k,s}
}
and
\bas{
C_{k,r}=\sqrt{(1-\beta) \beta^{-k}}   \frac1M  X_{k,r}
}
and observe that
\bas{\label{eq:23578563-2}
\sqrt{(1-\beta) \beta^{-k}}  \,\E\bigl[ h(\mathbb V_M) \sfrac 1M X_{k,r}\bigr]=\E\bigl[h(A_{k,r}+(B_{k,r}+C_{k,r})^2) C_{k,r}\bigr].
}
For every choice of $k$ and $r$, the random variables $A_{k,r}$, $B_{k,r}$ and $C_{k,r}$ are independent and $B_{k,r}$ and $C_{k,r}$ have mean zero. Therefore,
we can apply Lemma~\ref{le:3465}  and get that
\bas{\label{eq782356-2}
\bigl|\E\bigl[&h(A_{k,r}+(B_{k,r}+C_{k,r})^2) C_{k,r}\bigr]- \E\bigl[h'(A_{k,r})\,C_{k,r}^3\bigr]\bigr|\\
&\le \E[ h''(A_{k,r})] \bigl(2\E[|B_{k,r}|^3]\,\E[C_{k,r}^2]+5\E [B_{k,r}^2] \,\E[|C_{k,r}|^3]+\E[|C_{k,r}|^5]\bigr)}
Note that $A_{k,r}$ stochastically dominates $\beta \IV_M$ so that 
\bas{
\E[|h'(A_{k,r})|]\le \E[|h'(\beta \IV_M)|]\text{ \ and \ }
\E[h''(A_{k,r})]\le \E[h''(\IV_M)].
}
Moreover, for every $q\in \{2,3\}$ one has
\bas{
\E[|B_{k,r}|^q]&=((1-\beta)\beta^{-k})^{q/2}\,\E\Bigl[\Bigl|\frac 1M\sum_{s\in\{1,\dots,M\}\backslash\{r\}} X_{0,s}\Bigr|^q\Bigr]\\
&\le ((1-\beta)\beta^{-k})^{q/2}\,\E\Bigl[\Bigl|\frac 1M\sum_{s=1}^M X_{0,s}\Bigr|^q\Bigr]\\
&\le  ((1-\beta)\beta^{-k})^{q/2}\,C_q\,\E\Bigl[\Bigl|\frac 1{M^2}\sum_{s=1}^M X_{0,s}^2\Bigr|^{q/2}\Bigr]\\
&\le  ((1-\beta)\beta^{-k})^{q/2}\,C_q\,\frac 1{M^{q/2}}\E\bigl[| X_{0,1}|^q\bigr],
}
where $C_2=1$ and $C_3$ is the constant appearing in the Burkholder-Davis-Gundy inequality when applied for the moment $3$.
Next, observe that
\bas{
\E[|C_{k,r}|^q]&=((1-\beta)\beta^{-k})^{q/2}\, \frac1{M^q} \E\bigl[| X_{0,1}|^q\bigr].
}
Combining the previous three estimates with~(\ref{eq782356-2}) we get that
\bas{\bigl|\E\bigl[&h(A_{k,r}+(B_{k,r}+C_{k,r})^2) C_{k,r}\bigr]\big|\\
&\le \E[|h'(\beta \IV_M)|]\,  ((1-\beta)\beta^{-k})^{3/2}\,\varphi_3 M^{-3}\\
&\qquad+  \E[h''(\beta \IV_M)] ((1-\beta)\beta^{-k})^{5/2} (2 \kappa_3 \varphi_2 \varphi_3 M^{-7/2}+ 5  \varphi_2 \varphi_3 M^{-4}+\varphi_5 M^{-5}).
}
With (\ref{eq:23578563-1}) and (\ref{eq:23578563-2}) we get that
\bas{
|f_M(\theta)|&\le (1-\alpha) \sum _{k\in-\N_0}\alpha^{-k}\sqrt{(1-\beta)\beta^{-k}}^{-1}  
M \, \E\bigl[ h(\IV_M) \sfrac 1M X_{k,1}\bigr]\\
&\le  (1-\alpha) \sum _{k\in-\N_0}\alpha^{-k} (1-\beta)\beta^{-k} \varphi_3 \,\E[|h'(\beta \IV_M)|]\, M^{-2} \\
&\qquad +  (1-\alpha) \sum _{k\in-\N_0}\alpha^{-k} ((1-\beta)\beta^{-k} )^2  \,\E[h''(\beta \IV_M)] \\
&\qquad\qquad\qquad\qquad(2 C_3 \varphi_2 \varphi_3 M^{-5/2}+ 5  \varphi_2 \varphi_3 M^{-3}+\varphi_5 M^{-4})\\
&\le \beta^{-3/2} \frac{(1-\alpha)(1-\beta)}{1-\alpha\beta}\varphi_3\, \E[|h'( \IV_M)|]\,M^{-2} \\
&\qquad + \beta^{-5/2}\frac{(1-\alpha)(1-\beta)^2}{1-\alpha\beta^2}\E[h''( \IV_M)] (2 C_3 \varphi_2 \varphi_3 M^{-5/2}+ 5  \varphi_2 \varphi_3 M^{-3}+\varphi_5 M^{-4}),
}
where we have used in the las step that for $x\in(0,\infty)$, $|h'(\beta x)|\le \beta^{-3/2} |h'(\beta x)|$ and $|h''(\beta x)|\le \beta^{-5/2} |h''(\beta x)|$.
\end{proof}

We will provide a further lemma which will later allow us to control 
the terms $\E[|h'( \IV_M^{(i)})|]$ and $\E[h''( \IV_M^{(i)})]$ 
appearing in \cref{theo:732846}.

\begin{lemma}\label{le:2341}
Let $\beta,\delta,p,q\in(0,\infty)$ with $q<\beta^p$ and $\beta<1$ and let $\mathcal Z=(Z_{k})_{k\in-\N_0}$ be a sequence of independent real-valued random variables satisfying for every $k\in-\N_0$ that
\bas{\label{eq235}
\P(Z_{k}^{2}< \delta)\le q.
}
One has
  \bas{ \E\Bigl[\frac1{ v(\mathcal Z)^p}\Bigr]\le 
  \Bigl(\frac\beta{1-\beta}\Bigr)^p\frac {1-q}{\beta^p-q} \,\delta^{-p},
} 
where \bas{v:\R^{-\N_0}\backslash \{0\} \to [0,\infty), \ \bx\mapsto  (1-\beta) \sum_{k\in -\N_{0}} \beta^{-k} x_{k}^{2}.}
\end{lemma}

\begin{proof}
As consequence of the assumptions on $\mathcal Z$ one can couple $\mathcal Z$ with a sequence of independent Bernoulli distributed random variables $(I_k)_{k\in\N_0}$ with success probability $1-q$ such that for all $k\in\N_0$, 
\bas{\1_{\{Z_{-k}^2\ge\delta\}}\ge I_k.
}
We consider the geometrically distributed stopping time
\bas{
T=\inf\{k\in\N_0:   I_k=1\}.
}
One has that
$
v(\mathcal Z)\ge  (1-\beta) \beta^T \delta
$
so that
\bas{
\E[v(\mathcal Z)^{-p}]\le ((1-\beta)\delta)^{-p} \,\E[\beta^{-pT}] =((1-\beta)\delta)^{-p} \,\frac{1-q}{1-\beta^{-p}q}.
}
\end{proof}

\section{Proof of the theorem in the introduction (\cref{thm-1} in \cref{sec:introduction})}
\label{sec:proof_theorem_introduction}

In this section, we prove Theorem~\ref{thm-1}. We fix $d=1$ and denote by $U$  a  uniformly bounded random variable. For every $M\in\N$, let  $(\IX_M,\IU)$ be the innovation with $\IX_M \colon \R^\infty\times \R\to\R$ given by
\begin{equation}
\IX_M(u,\theta)=\frac 1M\sum_{i=1}^M u^{(i)}-\theta
\end{equation}
and $\IU=(U^{(i)})_{i\in\N}$ being a sequence of independent $U$-distributed random variables.
Let $(\theta^{[M]}_n)_{n\in\N_0}$ be an Adam algorithm with innovation $(\IX_M,\IU)$ and denote by $(\IU_k)_{k\in-\N_0}$ a sequence of independent copies of $\IU$ driving the Adam algorithm. We can represent the 
related Adam vector field $ f_M \colon \R \to \R $ as
\bas{
f_M(\theta)= (1-\alpha) \E\Bigl[\frac1{\epsilon + \sqrt{(1-\beta)\sum_{k\in-\N_{0}}\beta^{-k}\IX_M(\IU_{k},\theta)^{2}} } \sum_{k\in-\N_{0}}\alpha^{-k}     \IX_M(\IU_{k},\theta)\Bigr].
}
For $\rho\in[0,1)$ and $\theta\in\R$, let
\bas{
\mathbb M_M^\rho(\theta)=(1-\rho)\sum_{k\in-\N_{0}}\rho^{-k}\IX_M(\IU_k,\theta)\text{ \ and \ }  \mathbb V_M(\theta)=(1-\beta)\sum_{k\in-\N_{0}}\beta^{-k}\IX_M(\IU_{k},\theta)^{2}.
}
Provided that we can change differentiation and integration we get that
\bas{\label{eq8346}
f'_M(\theta)=- \E\Bigl[\frac1{\epsilon+\sqrt{\IV_M({\theta})}} \Bigr] +\E\Bigl[ \IM^\alpha_M({\theta})\IM^\beta_M({\theta})\frac 1{(\epsilon+\sqrt{\IV_M(\theta)})^{2}}\frac 1{\sqrt{\IV_M({\theta})}}\Bigr].
}
As we will verify below (see~(\ref{eq:2354681})) the term in both expectations is uniformly bounded and continuous which justifies the latter identity.

We first verify that on an arbitrary compact set the vector field $f_M$ is decreasing as long as $M$ is sufficiently large.

\begin{lemma}\label{le:12472}
 Let $V\subset \R$ be a compact set. There exists $M_0\in\N$ such that \bas{
\sup\{ f'_M(\theta): M\ge M_0, \theta\in V\} <0.
} 
\end{lemma}

We thus showed that for $M\ge M_0$ and the chosen $V$ one has that for every $x,y\in V$ with $x\ge y$, one has
\bas{
\langle f_M(x)-f_M(y),x-y\rangle =\int _y^x f_M'(r) (x-y)\, \dd r\le -c_1 (x-y)^2
}
which implies monotonicity of the Adam vector field.

\begin{proof}
We show that the integrands in the expectations of~(\ref{eq8346}) are uniformly bounded. 
By the Cauchy-Schwarz inequality, one has that
\bas{
|\IM_M^\alpha(\theta)|&=(1-\alpha)\Bigl|\sum_{k\in-\N_0} \alpha^{-k}\beta ^{k/2} \beta^{-k/2} \IX_M(\IU_k,\theta)\Bigr| \\
&\le (1-\alpha)\Bigl(\sum_{k\in-\N_0} \alpha^{-2k}\beta ^{k}\Bigr)^{1/2}\Bigl(\sum_{k\in-\N_0} \beta ^{-k} \IX_M(\IU_k,\theta)^2\Bigr)^{1/2}\\
&=  \frac{1-\alpha }{\sqrt{1-\alpha^2/\beta}\sqrt{1-\beta}} \sqrt{\IV_M(\theta)}
}
and by Jensen's inequality that
\bas{
\bigl|\IM_M^\beta(\theta)\bigr| \le \Bigl((1-\beta)\sum_{k\in-\N_{0}}\beta^{-k}\IX_M(\IU_{k},\theta)^2\Bigr)^{1/2}=\sqrt{\IV_M(\theta)}.
}
This implies that 
\bas{\label{eq:2354681}
\Bigl|\IM^\alpha_M({\theta})\IM^\beta_M({\theta}) \frac 1{(\sqrt{\IV_M({\theta})}+\epsilon)^{2}}\frac 1{\sqrt{\IV_M({\theta})}}\Bigr|\le \frac{1-\alpha }{\sqrt{1-\alpha^2/\beta}\sqrt{1-\beta}} \frac1{\sqrt {\IV_M(\theta)}+\epsilon}.
}
Next, take a convergent series $(\vartheta_M)_{M\in\N}$ and denote by $\vartheta$ its limit. Then for $\rho\in\{\alpha,\beta\}$
\bas{
\lim_{M\to\infty} \IM^\rho_M(\vartheta_M)= \E[U]-\vartheta,\text{ \ in probability,}
} 
and
\bas{
\lim_{M\to\infty} \IV_M(\vartheta_M)= (\E[U]-\vartheta)^2 ,\text{ \ in probability.}
}
This entails that, in probability,
\bas{
\lim_{M\to\infty} \frac1{\sqrt{\IV_M(\vartheta_M)}+\epsilon}= \frac1{|\E[U]-\vartheta|+\epsilon} 
}
and
\bas{
\lim_{M\to\infty} \IM^\alpha_M({\vartheta_M})\IM^\beta_M({\vartheta_M}) \frac 1{(\sqrt{\IV_M({\vartheta_M})}+\epsilon)^{2}}\frac 1{\sqrt{\IV_M({\vartheta_M})}}= \frac{|\E[U]-\vartheta| }{(|\E[U]-\vartheta|+\epsilon)^2}.
}
With dominated convergence, we thus get that
\bas{\label{eq9823564}
\lim_{M\to\infty} f'_M(\vartheta_M)= -\frac {1}{|\E[U]-\vartheta|+\epsilon}\Bigl(1- \frac{|\E[U]-\vartheta|} {|\E[U]-\vartheta|+\epsilon}\Bigr)<0.
}

We prove the statement by contradiction. If the statement would not be true for every choice of $M_0$ 
 there would exist a sequence $(M_k)_{k\in\N}$ of integers diverging to infinity and a $V$-valued sequence $(\vartheta_k)_{k\in\N}$ with
\bas{
\limsup_{k\to\infty} f'_{M_k}(\vartheta_k) \ge 0.
} 
By compactness of $V$, we can assume that $(\vartheta_k)$ is convergent with a limit $\theta\in V$ (if the sequence is not convergent, we can instead work with a convergent subsequence). This contradicts our computations above, see~(\ref{eq9823564}).
\end{proof}

\begin{lemma}\label{le:12472-2} 
There exists $M_{0}\in\N$ and $\kappa\in(0,\infty)$ such that for every $M\in\N$ with $M\ge M_{0}$ there exists $\vartheta_{M}\in\R$ with $|\vartheta_{M}-\E[U]|\le \kappa M^{-1}$ and $f_M(\vartheta_M)=0$.
\end{lemma}

\begin{proof} Set $u=\E[U]$.
Without loss of generality we can assume that $ U$ has a strictly-positive variance. (Otherwise, $f_M(u)=0$ and we can choose $\vartheta_M=u$.) To apply Theorem~\ref{theo:732846} we need estimates for $\E[|h'(\mathbbm V_M)|]$ and $\E[h''(\mathbbm V_M)]$.

Pick $q\in(0,\beta^{3/2})$.
The classical central limit theorem implies that
\bas{
\sqrt M \,\IX_M(\IU,\theta) \Rightarrow \mathcal N(0, \var(U))
}
as $M\to\infty$. Thus we can choose $M_0'\in\N$ and $\delta'\in(0,\infty)$ so that for all $M\in\N$ with $M\ge M_0'$ one has
\bas{
\P(|\sqrt M \, X_M (U,\theta)|< \sqrt {\delta'})\le q \text{ \ or, equivalently, \ }
\P( X_M (U,\theta)^2< {\delta'}/M)\le q.
}
With Lemma~\ref{le:2341} we thus obtain that for $p\in\{1/2,3/2\}$
\bas{\label{eq:345861}
\E[\mathbb V_M^{-p}]\le \Bigl(\frac\beta{1-\beta}\Bigr)^p\frac{1-q}{\beta^p-q} (M/\delta')^p.
}
In particular, the latter expectations are finite and 
we are now in the position to apply
Theorem~\ref{theo:732846}: there exists a constant $\kappa_1$ only depending on $\alpha$, $\beta$ and the first five moments of $U$ so that for every $M\in\N$ with $M\ge M_0'$ one has that
\bas{
|f_M(u)|\le \kappa_1\bigl( \E[|h'(\mathbb V_M)|]\,M^{-2}+ \E[h''(\mathbb V_M)] \,M^{-5/2}\bigr).
}
Note that for all $x\in(0,\infty)$, $|h'(x)|\le \epsilon^{-2} x^{-1/2}$ and $h''(x)\le \epsilon^{-2} x^{-3/2}$ so that
\bas{
|f_M(u)|\le \kappa_1\epsilon^{-2} \bigl( \E[\mathbb V_M^{-1/2}]\,M^{-2}+ \E[\mathbb V_M^{-3/2}] \,M^{-5/2}\bigr).
}
In view of (\ref{eq:345861}) we get  that
\bas{|f_M(u)|\le \kappa_2 \,M^{-1},
}
where $\kappa_2$ is a constant only depending on $\alpha,\beta,\epsilon,\delta'$  and $q$ and the first five moments of $U$.

By Lemma~\ref{le:12472}, we can choose  $M_0''\in\N$  such that 
\bas{
-r:=\sup\{ f'_M(x): x\in[u-1,u+1]\text{ and } M\ge M_0''\}<0.
}
Now choose $M_0 =\max \{M_0', M_0'',\kappa_2/r \}$ and note that for $M\in\N$ with $M\ge M_0$ one has
\bas{
f_M(u+\sfrac{\kappa_2}r M^{-1})\le |f_M(u)|-r \sfrac{\kappa_2}r M^{-1}\le 0
}
and 
\bas{
f_M(u-\sfrac{\kappa_2}r M^{-1})\ge -|f_M(u)|+r \sfrac{\kappa_2}r M^{-1}\ge 0.
}
By the intermediate value theorem, there exists $\vartheta_M\in [u-\sfrac{\kappa_2}r M^{-1},u+\sfrac{\kappa_2}r M^{-1}]$ with $f_M(\vartheta_M)=0$.
\end{proof}

\begin{proof}[Proof of Theorem~\ref{thm-1}]
Recall that by assumption, the random variable $U$ is uniformly bounded, say by $C_1\in(0,\infty)$.
By \cite[Corollary 2.5]{DereichGraeberJentzen2024arXiv_non_convergence} there exists a constant $C_2$ such that for every $M\in\N$, the respective Adam algorithm $(\theta_n^{[M]})$  satisfies for all $n\in\N_0$
\bas{
|\theta^{[M]}_n|\le C_2.
}
This entails that $X_n^{[M]}:=\IX_M(\IU_n,\theta^{[M]}_{n-1})$ $(n\in\N)$ is bounded by $C_1+C_2$ so that for $\bX^{[M]}_n=(\1_{\{n+k>0\}} X_{n+k}^{[M]})_{k\in-\N_0}$ one has that
$\|\bX^{[M]}_n\|_{\ell_\varrho}\le (C_1+C_2)\|\varrho\|_{\ell_1}$.

By Lemmas~\ref{le:12472} and~\ref{le:12472-2}, there exists $M_0\in\N$, $\kappa\in(0,\infty)$ and a $\R$-valued sequence $(\vartheta_M)_{M\ge M_0}$ such that $f_M(\vartheta_M)=0$, $|\vartheta_M-\E[U]|\le (\kappa M^{-1})\wedge C_2$ and
\bas{
-c_1:=\sup\bigl\{f_M'(\theta): M\ge M_0, \theta\in[-C_2,C_2]\bigr\}<0.
}
This entails that for every $\theta\in[-C,C]$,
\bas{\label{eq:934657}
(f_M(\theta)-f_M(\vartheta_M))(\theta-\vartheta_M)\le -c_1\,(\theta-\vartheta_M)^2.
}

We will apply Theorem~\ref{thm-2} with $\alpha,\beta,\epsilon$ and $(\gamma_n)$ as in the statement of Theorem~\ref{thm-1}. Moreover, let $c_1$ as above, $c_2=0$, $p\in(2,\infty)$, $c=c_1/2$, $\varepsilon=1$ and $\mathcal K=C_1+C_2+1$.

We choose $\mathfrak n\in\N_0$  and $\eta\in(0,\infty)$ with $\sqrt{\gamma_{\mathfrak n+1}}\,(C_1+C_2)\|\varrho\|_{\ell_1}  \le 1$ so that the statement in Theorem~\ref{thm-2} is true for the chosen set of parameters. 
 Fix $M\in\N$ with $M\ge M_0$ and let for $\zeta\in\R$, $m\in\R$ and $\nu\in[0,\infty)$, 
$\theta^{[\zeta,\mathfrak n,m,\nu]}=(\theta_n^{[M,\zeta,\mathfrak n,m,\nu]})_{n\ge \mathfrak n}$ denote the Adam algorithm with innovation $(\IX_M,\IU)$ started at time $\mathfrak n$ in $(\zeta,m,\nu)$.
Next, we will verify assumptions 1.-4. for appropriate $M$-dependent Adam algorithms with $M\ge M_0$. Fix $M\ge M_0$ and $V=[-C_2,C_2]$.

1.)  We have for every $\theta\in V$ that
\bas{
\E[|\IX_M(\IU,\theta)|^p]^{1/p}\le C_1+C_2\le\cK \text{ \ and \ } \E[|\IX_M(\IU,\theta)-\IX_M(\IU,\theta')|^p]^{1/p}=|\theta-\theta'|\le \cK |\theta-\theta'|.
}

2.) Note that for the random terms $m_{\mathfrak n}$ and $v_{\mathfrak n}$ one has 
\bas{ \sqrt{\gamma_{\mathfrak n+1}}\,(m_{\mathfrak n},v_{\mathfrak n})_{\ell_\varrho}\le \sqrt{\gamma_{\mathfrak n+1}}\, \|\bX^{[M]}({\mathfrak n})\|_{\ell_\varrho}\le \sqrt{\gamma_{\mathfrak n+1}}\, (C_1+C_2) \|\varrho\|_{\ell_1}\le1\le  \cK.
}

3.) We choose $\Psi_M:[t_{n_0},\infty)\to\R$ as constant function $\Psi^M\equiv \vartheta_M$ and note that it solves the ODE $\dot \Psi^M_t=0=f_M(\Psi^M_t)$.

4.) Choose $\mathfrak R_t\equiv\infty$ and note that for every $x\in V$, $(f(x)-f(\Psi_t^M))(x-\Psi_t^M)\le -c_1 \,|x-\Psi^M_t|^2$ is just property~(\ref{eq:934657}) above.

We conclude that for $n\ge \mathfrak n$
\bas{
\E[|\theta_n^{[M]}-\vartheta _M|^p]^{1/p}&= \E[\E[|\theta^{[M]}_n-\vartheta _M|^p|\cF_{\mathfrak n}]]\\
&= \E\Bigl[\E\bigl[\theta_n^{[M,\zeta,\mathfrak n,m,\nu]}-\vartheta_M|^p\bigr]\Big|_{(\zeta,m,\nu)=(\theta_{\mathfrak n},m_{n_0},v_{n_0})}  \Bigr]^{1/p}\\
&\le \E\Bigl[\Bigl(\eta +\Bigl((1+\varepsilon)\frac{|\theta_{\mathfrak n}-\Psi_{t_{n_{0}}}|}{\sqrt{\gamma_{\mathfrak n+1}}}+\eta \sqrt{(m_{\mathfrak n},v_{\mathfrak n})_{\ell_{\varrho}}}\Bigr) e^{-c (t_{n}-t_{\mathfrak n})}\Bigr)^p
\Bigr]^{1/p} \sqrt{\gamma_{n+1}}\\
&\le (\eta+(4C_2 \gamma_{\mathfrak n+1}^{-1/2}]+\eta  )e^{-c(t_n-t_{\mathfrak n})})\sqrt{\gamma_{n+1}}.
}
Moreover, $\E[|\theta_n^{[M]}-\vartheta|^p]^{1/p}\le 2C_2$ for all $n\in\N$. The latter two bounds do not depend on the choice of $M$ which proves estimate (b) for an appropriately chosen $\eta$.

To prove (a) we choose $p\in(2,\infty)$ with $\sum_{n\in\N_0} \gamma_n^{p/2}<\infty$. Using property (b) (with the respective $\eta$) we conclude that
\bas{
\E\Bigl[\sum_{n\in\N_0}|\theta_n^{[M]}-\vartheta_M|^p \Bigr] \le \eta^p \sum_{n\in\N_0} \gamma_n^{p/2}<\infty
}
which implies that, almost surely, $\sum_{n\in\N_0}|\theta_n^{[M]}-\vartheta_M|^p<\infty$ and that $(\theta^{[M]}_n)$ tends to $\vartheta_M$, almost surely.
\end{proof}

\subsection*{Acknowledgements}
Shokhrukh Ibragimov is gratefully acknowledged for having brought 
useful related research findings to our attention. 
This work has been partially funded by the European Union (ERC, MONTECARLO, 101045811). 
The views and the opinions expressed in this work are however those of the authors only 
and do not necessarily reflect those of the European Union 
or the European Research Council (ERC). Neither the European Union nor 
the granting authority can be held responsible for them. 
In addition, this work has been partially funded by the Deutsche Forschungsgemeinschaft (DFG, German Research Foundation) 
under Germany's Excellence Strategy EXC 2044--390685587, Mathematics Münster: Dynamics--Geometry--Structure.

\bibliographystyle{acm}
\bibliography{bibfile}

\end{document}